\documentclass[10pt,a4paper]{article}
\usepackage[utf8]{inputenc}
\usepackage{amsmath}
\usepackage{amsthm}
\usepackage{amsfonts}
\usepackage{algorithm}
\usepackage{algpseudocode}
\newcommand{\hlabel}{\phantomsection\label}

\usepackage{multirow}
\usepackage[table,xcdraw]{xcolor}
\usepackage{subcaption}
\usepackage{float}
\usepackage{rotating}
\usepackage{soul}
\usepackage{hyperref}
\usepackage{tabularx}
\usepackage{multicol}
\usepackage{tikz}
\usepackage{amssymb}
\usepackage{blindtext}
\usepackage{graphicx}
\usepackage[shortlabels]{enumitem}
\usepackage[left=2.5cm,right=2.5cm,top=2cm,bottom=2cm]{geometry}
\usepackage{booktabs}
\usepackage[table,xcdraw]{xcolor}
\usepackage{multirow}
\usepackage{rotating,tabularx}
\usepackage{longfbox}

\newtheorem{Step}{Step}
\newtheorem{Claim}{Claim}
\newtheorem{Theorem}{Theorem}[section]
\newtheorem{Proposition}[Theorem]{Proposition}
\newtheorem{Remark}[Theorem]{Remark}
\newtheorem{Lemma}[Theorem]{Lemma}
\newtheorem{Corollary}[Theorem]{Corollary}
\newtheorem{Definition}[Theorem]{Definition}
\newtheorem{Example}{Example}
\newtheorem{Algorithm}{Algorithm}
\usepackage{hyperref}
\hypersetup{colorlinks=true,urlcolor=red}
\expandafter\let\expandafter\oldproof\csname\string\proof\endcsname
\let\oldendproof\endproof
\renewenvironment{proof}[1][\proofname]{
\oldproof[\ttfamily\scshape \bf #1.]
}{\oldendproof}
\newcommand{\set}[1]{\left\{#1\right\}}

\def\tilde{\widetilde}

\def\ox{\overline{x}}

\def\epsilon{\varepsilon}
\def\ox{\bar{x}}

\def \N{{\rm I\!N}}
\def \R{{\rm I\!R}}

\newcommand{\dotproduct}[1]{\left\langle#1\right\rangle}
\newcommand{\brac}[1]{\left(#1\right)}
\newcommand{\sbrac}[1]{\left[#1\right]}
\newcommand{\abs}[1]{\left|#1\right|}
\newcommand{\norm}[1]{\left\|#1\right\|}
\numberwithin{equation}{section}

\title{\bf General Derivative-Free Optimization Methods under Global and Local Lipschitz Continuity of Gradients}
\author{Pham Duy Khanh\footnote{Group of Analysis and Applied Mathematics, Department of Mathematics, Ho Chi Minh City University of Education, Ho Chi Minh City, Vietnam. E-mails: pdkhanh182@gmail.com, khanhpd@hcmue.edu.vn}\quad Boris S. Mordukhovich\footnote{Department of Mathematics, Wayne State University, Detroit, Michigan, USA. E-mail: aa1086@wayne.edu. Research of this author was partly supported by the US National Science Foundation under grants DMS-1808978 and DMS-2204519, by the Australian Research Council under grant DP-190100555, and by Project 111 of China under grant D21024.}\quad Dat Ba Tran\footnote{Department of Mathematics, Wayne State University, Detroit, Michigan, USA. E-mail: tranbadat@wayne.edu. Research of this author was partly supported by the US National Science Foundation under grants DMS-1808978 and DMS-2204519.}}
\begin{document}
\maketitle
\vspace*{-0.3in}
\noindent
{\small{\bf Abstract}. } This paper addresses the study of derivative-free smooth optimization problems, where the gradient information on the objective function is unavailable. Two novel general derivative-free methods are proposed and developed for minimizing such functions with either global or local Lipschitz continuous gradients. The newly developed methods use gradient approximations based on finite differences, where finite difference intervals are automatically adapted to the magnitude of the exact gradients without knowing them exactly. The suggested algorithms achieve fundamental convergence results, including stationarity of accumulation points in general settings as well as global convergence with constructive convergence rates when the Kurdyka-\L ojasiewicz property is imposed. The local convergence of the proposed algorithms to nonisolated local minimizers, along with their local convergence rates, is also analyzed under this property. Numerical experiences involving various convex, nonconvex, noiseless, and noisy functions demonstrate that the new methods exhibit essential advantages over other state-of-the-art methods in derivative-free optimization.\\[1ex] 
\section{Introduction}
This paper addresses developing novel {\em derivative-free methods} to solve unconstrained optimization problems given in the form
\begin{align}\label{optim prob}
{\rm minimize}\quad f(x)\quad\text{ subject to }\; x\in\R^n,
\end{align}
where $f:\R^n\rightarrow\R$  is a continuously differentiable ($\mathcal{C}^{1}$-smooth) function, and where only a {\em noisy approximation} $\phi(x)=f(x)+\xi(x)$ of the objective function is computable. It means that the {\em gradient} information of the objective function is fully {\em unavailable}, even in the noiseless case when the noise function $\xi(x)$ vanishes. These problems have received much attention with a variety of methods being developed over the years. The major developments in this vein are provided by the  {\em Nelder-Mead simplex method} \cite{lagarias98,nelder65}, {\em direct search methods} \cite{gray06,hooke61} and {\em trust-region methods} \cite{conn09 siam}. More recently, numerous empirical results conducted by Shi et al. in \cite{shi23} show that derivative-free optimization methods based on {\em finite differences} are accurate, efficient, and in some cases superior to other state-of-the-art derivative-free optimization methods established in the literature. Meanwhile, extensive numerical comparisons in Berahas et al. \cite{berahas22} together with the further analysis by Scheinberg \cite{scheinberg22} also tell us that the accuracy of gradients obtained from finite differences is much higher than from {\em randomized schemes}. These empirical results suggest that the methods using finite differences have much to be recommended, and that the research in this direction should be strongly encouraged. 

There are several derivative-free optimization methods in the literature that successfully employ finite differences. Kelley et al. \cite{choi00,gilmore95,kelly99} proposed the {\em implicit filtering} algorithm based on a finite difference approach to deal with  \textit{noisy smooth box-constrained} optimization problems. Berahas et al. \cite{berahas19,berahas21} developed finite difference-based line search methods for the minimization of {\em noisy smooth functions}. The approaches to compute the \textit{finite difference interval} are also studied by Gill et al. \cite{gill83}, and recently by Mor\'e and Wild \cite{more12} as well as by Shi et al. \cite{shi22}. Although methods of this type are often used in practice to solve derivative-free smooth problems, the {\em theoretical understanding} of behavior of their iterates in general cases is rather limited in comparison with analysis of classical first-order methods. {\em Stationarity of accumulation points} is a fundamental convergence property that is achieved by many first-order methods \cite{bertsekasbook,nesterovbook} and by some derivative-free optimization algorithms using the simplex-based direct method \cite[Theorem~3.2]{tseng99}, line search based on simplex derivatives \cite[Chapter~9]{conn09}, coordinate search with difference Hessians \cite[Theorem~4.2]{mifflin75}. However, to the best of our knowledge, this property has not been established for derivative-free methods employing \textit{finite differences} for general classes of smooth functions (i.e., smooth functions with global or local Lipschitz continuous gradients), even in noiseless cases. Note that the convergence properties of derivative-free methods for noisy functions studied in \cite{berahas21,berahas19} do not ensure them for noiseless functions. The reason is that the usual selection of finite difference intervals based on the noise level becomes zero, being invalid therefore in this case; see, e.g., \cite[Equations (3.3), (3.4)]{berahas19}. The implicit filtering algorithm \cite[Theorem~2.1]{choi00}, although achieving the stationarity for accumulation points, requires the objective function to be $\mathcal{C}^2$-smooth with rather strict conditions including the boundedness of level sets and the uniform boundedness of the inverses of Hessians. Furthermore, the convergence of the sequence of iterates to a {\em nonisolated} stationary point is also ignored in the analysis of aforementioned derivative-free optimization methods while obtained by various first-order methods \cite{attouch13,kmt22.1,kmt23.1,kmpt23.2} under the {\em Kurdyka-\L ojasiewicz} (KL) property \cite{attouch10,kurdyka,lojasiewicz65}, which is a rather mild regularity condition satisfied for the vast majority of objective functions in practice. Of course, the local convergence to nonisolated local minimizers for such derivative-free methods is still questionable.

The assumption of {\em global Lipschitz continuity}  of the gradient seems to be omnipresent in derivative-free line search methods; see, e.g., \cite[Theorem~2.1]{choi00}, \cite[Assumption~A1]{berahas19}, \cite[Assumption~1.1]{berahas21}.  For general derivative-free trust-region methods,  Conn et al. \cite{conn09 siam} prove global convergence results under the global Lipschitz continuity of the gradient and the Hessian of the objective function. Such properties are also employed in the proximal point method adapted to derivative-free smooth optimization problems  by Hare and Lucet \cite[Assumption 1]{hare14}. We are not familiar with any efficient finite difference-based method considering specifically the class of smooth functions with \textit{locally Lipschitz} continuous gradients, which is much broader than the class of smooth functions with globally Lipschitz continuous gradients and covers, in particular, $\mathcal{C}^2$-smooth functions. This is in contrast to the exact versions of gradient descent methods, which obtain various convergence properties including the stationarity of accumulation points for the version with backtracking stepsizes on the class of $\mathcal{C}^1$-smooth functions \cite[Proposition~1.2.1]{bertsekasbook} and the global convergence for the version with sufficiently small stepsizes on the class of definable $\mathcal{C}^{1}$-smooth functions with locally Lipschitz continuous gradients \cite{josz23}. This raises the need for the design and analysis of finite difference-based methods concerning the class of smooth functions with locally Lipschitz continuous gradients, which is the best we can hope in this context since the error bounds for finite differences are not available outside this class.

Having in mind the above discussions, we propose in this paper {\em two novel derivative-free methods} for minimizing smooth functions. A general derivative-free method with {\em constant stepsize} is designed for smooth functions whose gradients are {\em globally} Lipschitzian although the exact calculation of the Lipschitz constant is not necessary. A general derivative-free method with {\em backtracking stepsize} is designed for smooth functions whose gradients are only {\em locally} Lipschitzian. Such newly developed methods are analyzed rigorously for {\em noiseless functions}, which particularly have many applications in the design of space trajectories \cite{addis11}, reinforcement learning \cite{malik19}, and other black-box optimization problems \cite{audet17,marsden07}. For such problems, the objective functions may in fact well-behave, even being smooth, while the exact gradient information is either unavailable due to the lack of analytical forms of the objective function, or it is computationally expensive to obtain. The main ideas for our methods come mostly from the {\em inexact reduced gradient methods} (IRGM) introduced and justified in \cite{kmt22.1}. Roughly speaking, the newly developed derivative-free methods are designed as inexact gradient descent methods for smooth functions and thus inherit typical convergence properties for first-order methods including stationarity of accumulation points, {\em global convergence} under the KL property, and constructive convergence rates depending on the KL exponent. Furthermore, we establish {\em local convergence} to {\em nonisolated} local minimizers for both proposed methods under the KL property of the objective function.

Regarding {\em numerical aspects}, our methods borrow the ideas of {\em gradient approximations} from IRGM \cite{kmt22.1} and its nonsmooth extensions \cite{kmt23.1,kmpt23.2}, where errors are automatically adapted to the magnitude of exact gradients; see, e.g., \cite[Section~6.2]{kmt22.1}. Correspondingly, the proposed derivative-free methods generate at each iteration a descent direction based on an approximate gradient associated with the largest finite difference interval. Choosing such an interval in this way allows the methods to omit \textit{rounding errors} as much as possible, which makes them more efficient without knowing the \textit{noise level} in advance. This selection of the finite difference interval is also completely different from the selection used in the implicit filtering algorithm \cite{choi00,gilmore95,kelly99}, since the latter forces the finite difference interval to decrease after each iteration while ours does not. Extensive {\em empirical results} on diverse datasets including convex, nonconvex, noiseless, and noisy functions with different noise levels in different finite-dimensional spaces show that our newly developed methods usually perform better than many other state-of-the-art derivative-free methods in solving smooth problems. For the conciseness of the paper, the convergence analysis and numerical experiments involving methods coupled with quasi-Newton algorithms are not considered here but are deferred to our future research.

The rest of the paper is organized as follows. Section~\ref{Sec:2} presents some basic definitions and preliminaries used throughout the entire paper. Section~\ref{sec:3} examines two types of approximation for gradients that cover finite differences.  The main parts of our work concerning the design and convergence properties of general derivative-free methods under the global and local Lipschitz continuity of the gradient are given in Section~\ref{sec:4} and Section~\ref{sec:5}, respectively. Local convergence of the proposed methods to nonisolated local minimizers is presented in Section~\ref{sec:6}. Numerical experiments comparing the efficiency of the proposed methods with other derivative-free methods for both noisy and noiseless functions are conducted in Section~\ref{sec:7}. Concluding remarks on the contribution of this paper together with some open questions and perspectives of our future research are given in Section~\ref{sec:concl}.\vspace*{-0.1in}

\section{Preliminaries}\label{Sec:2}

First we recall some basic notions and notation frequently used in the paper. All our considerations are given in the space $\R^n$ with the Euclidean norm $\|\cdot\|$. For any $i=1,\ldots,n$, let $e_i$ denote the $i^{\rm th}$ basic vector in $\R^n.$ As always, $\N:=\{1,2,\ldots\}$ signifies the collection of natural numbers. For any $x\in\R^n$ and $\varepsilon>0,$ let $\mathbb{B}(x,\varepsilon)$ and $\overline{\mathbb{B}}(x,\varepsilon)$ stand for
the open and closed balls centered at $x$ with radius $\varepsilon$, respectively. When $x=0$, these balls are denoted simply by $\varepsilon\mathbb{B}$ and $\varepsilon\overline{\mathbb{B}}$.

Recall  
that a mapping $G\colon\R^n\to\R^m$ is {\em Lipschitz continuous on a subset} $D$ of $\R^n$ if there exists a constant $L>0$ such that we have
\begin{align*}
\norm{G(x)-G(y)}\le L\norm{x-y} \text{ for all }x,y\in D.
\end{align*}
If $D=\R^n$, the mapping $G$ is said to be {\em globally Lipschitz continuous}. The  {\em local Lipschitz continuity} of $G$ on $\R^n$ is understood as the Lipschitz continuity of this mapping on every compact subset of $\R^n$. The latter is equivalent to saying that for any $x\in\R^n$ there is a neighborhood $U$ of $x$ such that $G$ is Lipschitz continuous on $U$, which is also equivalent to the Lipschitz continuity of $G$ on every bounded subset of $\R^n$. In what follows, we denote by $\mathcal{C}^{1,1}$ the class of ${\cal C}^1$-smooth mappings that have a {\em locally Lipschitz continuous gradient} on $\R^n$ and by $\mathcal{C}^{1,1}_L$ the class of ${\cal C}^1$-smooth mappings that have a {\em globally Lipschitz continuous gradient with the constant} $L>0$ on the entire space.\vspace*{0.03in}

Our convergence analysis of the numerical algorithms developed in the subsequent sections takes advantage of the following important results and notions. The first result taken from \cite[Lemma~A.11]{solodovbook} presents a simple albeit very useful property of real-valued functions with Lipschitz continuous gradients.

\begin{Lemma}\label{lemma descent}
Let $f:\R^n\rightarrow\R$, and let $x,y\in\R^n$. If $f$ is differentiable on the line segment $[x,y]$ with its derivative being Lipschitz continuous on this segment with a constant $L>0$, then
\begin{align}
\abs{f(y)-f(x)-\dotproduct{\nabla f(x),y-x}}\le \dfrac{L}{2}\norm{y-x}^2.
\end{align}
\end{Lemma}
The second lemma established in \cite[Section~3]{kmt22.1} is crucial in the convergence analysis of the general linesearch methods developed in this paper.

\begin{Lemma}\label{stationary point lemma}
Let $\set{x^k}$ and $\set{d^k}$ be sequences in $\R^n$ satisfying  the condition
\begin{align}\label{series is finite}
\sum_{k=1}^\infty\norm{x^{k+1}-x^k}\norm{d^k}<\infty.
\end{align} 
If $\bar x$ is an accumulation point of $\set{x^k}$ and if the origin is an accumulation point of $\set{d^k}$, then there exists an infinite set $J\subset\N$ such that
\begin{align}\label{relation dk xk 3}
x^k\overset{J}{\rightarrow}\bar x\;\mbox{ and }\;d^k\overset{J}{\rightarrow}0.
\end{align}
\end{Lemma}
Next we recall the classical results from \cite[Section~8.3.1]{pangbook2} that describe important properties of accumulation points
generated by a sequence satisfying a limit condition introduced by Ostrowski \cite{ostrowski}.
\begin{Lemma}\label{lemma: ostrowski}
Let $\set{x^k}\subset\R^n$ be a sequence satisfying the Ostrowski condition
\begin{align}\label{8.3.8}
\lim_{k\rightarrow\infty}\|x^{k+1}-x^k\|=0.
\end{align}
Then the following assertions are fulfilled:

{\bf(i)} If $\set{x^k}$ is bounded, then the set of accumulation points of $\set{x^k}$ is nonempty, compact, and connected in $\R^n$.

{\bf(ii)} If $\set{x^k}$ has an isolated accumulation point, then this sequence converges to it.
\end{Lemma}
The version of the {\em Kurdyka-\L ojasiewicz property} below is taken from Absil et al. \cite[Theorem~3.4]{absil05}.
\begin{Definition}\rm \label{KL ine}\rm
Let $f:\R^n\rightarrow\R$ be a differentiable function. We say that $f$ satisfies the \textit{KL property} at $\bar{x}\in\R^n$ if there exist a number $\eta>0$, a neighborhood $U$ of $\bar{x}$, and a nondecreasing function $\psi:(0,\eta)\rightarrow(0,\infty)$ such that the function $1/\psi$ is integrable over $(0,\eta)$ and we have
\begin{align}\label{kl absil}
 \norm{\nabla f(x)}\ge\psi\big(f(x)-f(\bar{x})\big)\;\mbox{ for all }\;x\in U\;\mbox{ with }\;f(\bar x)<f(x)<f(\bar x)+\eta.
\end{align}
\end{Definition}
\begin{Remark}\rm\label{algebraic} If $f$ satisfies the KL property at $\bar x$ with a neighborhood $U$, it is clear that the same property holds for any $x\in U$ where $f(x)=f(\bar x)$. It has been realized that the KL property is satisfied in broad settings. In particular, it holds at every {\em nonstationary point} of $f$; see \cite[Lemma~2.1~and~Remark~3.2(b)]{attouch10}.  Furthermore, it is proved in the seminal paper by \L ojasiewicz \cite{lojasiewicz65} that any analytic function $f:\R^n\rightarrow\R$ satisfies the KL property at every point $\ox$ with $\psi(t)~=~Mt^{q}$ for some $q\in [0,1)$. As demonstrated in \cite[Section~2]{kmt22.1}, the KL property formulated in Attouch et al. \cite{attouch10} is stronger than the one in Definition~\ref{KL ine}. Typical smooth functions that satisfy the KL property from \cite{attouch10}, and hence the one from Definition~\ref{KL ine}, are smooth {\em semialgebraic} functions and also those from the more general class of functions known as {\em definable in o-minimal structures}; see \cite{attouch10,attouch13,kurdyka}. The latter property is fulfilled, e.g., in important models arising in deep neural networks, low-rank matrix recovery, principal component analysis, and matrix completion as discussed in \cite[Section~6.2]{bolte21}.
\end{Remark}
Next we present, based on \cite{absil05}, some descent-type conditions ensuring the global convergence of iterates for smooth functions that satisfy the KL property. 
\begin{Proposition}\label{general convergence under KL}
Let $f:\R^n\rightarrow\R$ be a $\mathcal{C}^1$-smooth function, and let the sequence of iterations $\set{x^k}\subset\R^n$ satisfy the following conditions:
\begin{itemize}
\item[\bf(H1)] {\rm(primary descent condition)}. There exists $\sigma>0$ such that for sufficiently large $k\in\N$ we have 
\begin{align*}
 f(x^k)-f(x^{k+1})\ge\sigma\norm{\nabla f(x^k)}\norm{x^{k+1}-x^k}.
\end{align*}
\item[\bf(H2)] {\rm(complementary descent condition)}. For sufficiently large $k\in\N$, we have
\begin{align*}
 \big[f(x^{k+1})=f(x^k)\big]\Longrightarrow [x^{k+1}=x^k].
\end{align*}
\end{itemize}
If $\bar x$ is an accumulation point of $\set{x^k}$  and $f$ satisfies the KL property at $\bar x$, then $x^k\rightarrow\bar x$ as $k\to\infty$.
\end{Proposition}
When the sequence under consideration is generated by a linesearch method and satisfies some conditions stronger than (H1) and (H2) in 
Proposition~\ref{general convergence under KL}, its convergence rates  are established in \cite[Proposition~2.4]{kmt22.1} under the KL property with $\psi(t)=Mt^{q}$ as given below.
\begin{Proposition}\label{general rate}
Let $f:\R^n\rightarrow\R$ be a $\mathcal{C}^1$-smooth function, and let the sequences $\set{x^k}\subset\R^n,\set{\tau_k}\subset[0,\infty),\set{d^k}\subset\R^n$ satisfy the iterative condition $x^{k+1}=x^k+\tau_kd^k$ for all $k\in\N$. Assume that for sufficiently large $k\in\N$, we have $x^{k+1}\ne x^k$ and the estimates
\begin{align}\label{two conditions}
f(x^k)-f(x^{k+1})\ge \beta \tau_k\norm{d^k}^2\;\text{ and }\;\norm{\nabla f(x^k)}\le \alpha\norm{d^k},
\end{align}
where $\alpha,\beta$ are some positive constants. Suppose in addition that the sequence $\set{\tau_k}$ is bounded away from $0$ $($i.e., there is some $\bar \tau>0$ such that $\tau_k\ge \bar \tau$ for sufficiently large $k\in\N)$, that $\bar x$ is an accumulation point of $\set{x^k}$, and that $f$ satisfies the KL property at $\bar x$ with $\psi(t)=Mt^{q}$ for some $M>0$ and $q\in(0,1)$. Then the following convergence rates are guaranteed:
\begin{itemize}\vspace*{-0.05in}
\item[\bf(i)] If $q\in(0,1/2]$, then the sequence $\set{x^k}$ converges linearly to $\bar x$.\vspace*{-0.05in}
\item[\bf(ii)]If $q\in(1/2,1)$, then we have the estimate
\begin{align*}
\norm{x^k-\bar x}=\mathcal{O}\brac{ k^{-\frac{1-q}{2q-1}}}.
\end{align*}
\end{itemize}
\end{Proposition}
\begin{Remark}\rm \hlabel{rmk two KL}
Observe that the two conditions in \eqref{two conditions} together with the boundedness away from $0$ of $\set{\tau_k}$ yield assumptions (H1), (H2) in Proposition~\ref{general convergence under KL}. Indeed, (H1) is verified by the following inequalities:
\begin{align*}
f(x^k)-f(x^{k+1})&\ge \beta\tau_k\norm{d^k}^2=\beta\norm{\tau_k d^k}\norm{d^k}\\
&\ge \dfrac{\beta}{\alpha}\norm{x^{k+1}-x^k}\norm{\nabla f(x^k)}.
\end{align*}
In addition, since $\{\tau_k\}$ is bounded away from $0$, there is $\bar\tau>0$ such that $\tau_k\ge \bar \tau$ for sufficiently large $k\in\N$. Then for such $k$, the condition $f(x^{k+1})=f(x^k)$ implies that $d^k=0$  by the first inequality in \eqref{two conditions}, and hence $x^{k+1}=x^k$ by the iterative procedure $x^{k+1}=x^k+\tau_kd^k.$ This also verifies (H1).
\end{Remark}\vspace*{-0.2in}

\section{Global and Local Approximations of Gradients}\label{sec:3}

This section is devoted to analyzing several methods for approximating gradients of a smooth function $f:\R^n\rightarrow\R$ by using only information about function values that frequently appears in derivative-free optimization. Methods of this type include, in particular, finite differences \cite[Section~9]{nocedalbook}, the Gupal estimation \cite{hare13}, and gradient estimation via linear interpolation \cite{berahas22}. We construct two types of approximations that cover all these methods.
\begin{Definition}\rm\label{defi approx}
Let $f:\R^n\rightarrow\R$ be a $\mathcal{C}^1$-smooth function. A mapping $\mathcal{G}:\R^n\times (0,\infty)\rightarrow\R^n$ is :

{\bf(i)} A {\em global approximation} of $\nabla f$ if there is a constant $C>0$ such that 
\begin{align}\label{global approx}
\norm{\mathcal{G}(x,\delta)-\nabla f(x)}\le C\delta\;\text{ for any }\;(x,\delta)\in\R^n\times (0,\infty).
\end{align}

{\bf(ii)} A {\em local approximation} of $\nabla f$ if for any bounded set $\Omega\subset\R^n$ and any $\Delta>0$, there is $C>0$ with
\begin{align}\label{local approx}
\norm{\mathcal{G}(x,\delta)-\nabla f(x)}\le C\delta\;\text{ for any }\; (x,\delta)\in\Omega\times (0,\Delta].
\end{align}
\end{Definition}
\begin{Remark}\rm\label{rmk error bound}
We have the following observations related to Definition \ref{defi approx}:

{\bf(i)} If $\mathcal{G}$ is a global approximation of $\nabla f$, then it is also a local approximation of $\nabla f$.

{\bf(ii)} Assume that $\mathcal{G}$ is a local approximation of $\nabla f$ and that $x\in\R^n$. Then we deduce from \eqref{local approx} with $\Omega=\set{x}$ and any $\Delta>0$ the condition
\begin{align}\label{limsup bound}
\limsup_{\delta\downarrow0}\dfrac{\norm{\mathcal{G}(x,\delta)-\nabla f(x)}}{\delta}<\infty.
\end{align}
\end{Remark}
Next we recall the two standard types of finite differences  taken from \cite[Section~9]{nocedalbook}, which serve as typical examples of the approximations in Definition~\ref{defi approx}.\vspace*{0.05in}

$\bullet$ {\em Forward finite difference}: 
\begin{align}\label{forward}
\mathcal{G}(x,\delta):=\dfrac{1}{\delta}\sum_{i=1}^n\brac{f(x+\delta e_i)-f(x)}e_i\text{ for any }(x,\delta)\in\R^n\times (0,\infty).
\end{align}

$\bullet$ {\em Central finite difference}: 
\begin{align}\label{central}
\mathcal{G}(x,\delta):=\dfrac{1}{2\delta}\sum_{i=1}^n\brac{f(x+\delta e_i)-f(x-\delta e_i)}e_i\text{ for any }(x,\delta)\in\R^n\times (0,\infty).
\end{align}

\begin{Remark}\rm Let us now recall some results on the error bounds for the two types of finite differences that are mentioned above.

{\bf(i)} The global error bound for the forward finite difference (see, e.g., \cite[Theorem~2.1]{berahas22} and \cite[Section~8]{nocedalbook}) shows that it is a global approximation of $\nabla f$ when $f\in\mathcal{C}^{1,1}_L$. The local error bound for the forward finite difference is also formulated in \cite[Exercise~9.13]{nocedalbook}.

{\bf(ii)} On the other hand, the global error bound for the central finite difference (see, e.g. \cite[Theorem~2.2]{berahas22} and \cite[Lemma~9.1]{nocedalbook}) requires that $f$ is twice continuously differentiable with a Lipschitz continuous Hessian, which is a restrictive assumption. 
\end{Remark}

For completeness, we present a short proof showing that both types of finite differences are global approximations of $\nabla f$ when $f\in\mathcal{C}^{1,1}_L$ and are local approximations of $\nabla f$ when $f\in\mathcal{C}^{1,1}$.

\begin{Proposition}\label{for err bound}
Let $f:\R^n\rightarrow\R$ be a $\mathcal{C}^1$-smooth function. Then the following hold:

{\bf(i)} Given $x\in\R^n$ and $\delta>0$, if $\nabla f$ is Lipschitz continuous on $\overline{\mathbb{B}}(x,\delta)$ the with the constant $L>0$, then both the forward finite difference \eqref{forward} and the central finite difference \eqref{central} satisfy the estimate
\begin{align}\label{general ball error}
\norm{\mathcal{G}(x,\delta)-\nabla f(x)}\le \dfrac{L\sqrt{n}\delta}{2}.
\end{align}

{\bf(ii)} If $\nabla f$ is globally Lipschitz continuous with the constant $L>0$, then both the forward finite difference \eqref{forward} and the central finite difference \eqref{central} satisfy the estimate
\begin{align}\label{glo err bound for}
\norm{\mathcal{G}(x,\delta) -\nabla f(x)}\le \dfrac{L\sqrt{n}\delta}{2}\; \text{ for any }\;(x,\delta)\in\R^n\times(0,\infty).
\end{align}

{\bf(iii)} If $\nabla f$ is locally Lipschitz continuous, then for any bounded set $\Omega\subset\R^n$ and for any $\Delta>0$, there exists a positive number $L$ such that both the forward finite difference \eqref{forward} and the central finite difference \eqref{central} satisfy the estimate
\begin{align}\label{local err bound for}
\norm{\mathcal{G}(x,\delta) -\nabla f(x)}\le \dfrac{L\sqrt{n}\delta}{2}\;\text{ for any }\;(x,\delta)\in\Omega\times (0,\Delta].
\end{align}
\end{Proposition}
\begin{proof}
First we verify (i) for each type of the aforementioned finite difference and then employ (i) to justify (ii) and (iii) for both types.

(i) Take any $x\in\R^n,\delta>0$ and assume that $\nabla f$ is
Lipschitz continuous on $\overline{\mathbb{B}}(x,\delta)$ with the constant $L>0$. Consider first the case where $\mathcal{G}$ is given by the forward finite difference \eqref{forward}. Then for any $i=1,\ldots,n$ we get by employing Lemma~\ref{lemma descent} that
\begin{align}
&  \abs{f(x+\delta e_i)-f(x)-\dotproduct{\nabla f(x),x+\delta e_i-x}}\le \dfrac{L}{2}\norm{x+\delta e_i-x}^2=\dfrac{L\delta^2}{2},
\end{align}
which is clearly equivalent to
\begin{align*}
\abs{ \dfrac{1}{\delta}\brac{f\brac{x+\delta e_i}-f\brac{x}}-\dfrac{\partial f}{\partial x_i}(x)}\le \dfrac{L\delta}{2}.
\end{align*}
Since the latter inequality holds for all $i=1,\ldots,n$, we deduce that
\begin{align*}
\norm{\mathcal{G}(x,\delta)-\nabla f(x)}&=\sqrt{\sum_{i=1}^n \brac{\dfrac{1}{\delta}\brac{f(x+\delta e_i)-f\brac{x}}-\dfrac{\partial f}{\partial x_i}(x)}^2}\le \dfrac{L\sqrt{n}\delta}{2},
\end{align*}
which therefore verifies  estimate \eqref{general ball error}.

Assume now that $\mathcal{G}$ is given by the central finite difference \eqref{central}. Employing Lemma~\ref{lemma descent} gives us for any $i=1,\ldots,n$ the two estimates
\begin{align*}
& \abs{f(x+\delta e_i)-f(x)-\dotproduct{\nabla f(x),(x+\delta e_i)-x}}\le \dfrac{L\delta^2}{2},\\
& \abs{f(x)-f(x-\delta e_i)-\dotproduct{\nabla f(x),x-(x-\delta e_i)}}\le \dfrac{L\delta^2}{2}.
\end{align*}
Summing up the above estimates and using the triangle inequality, we deduce that
\begin{align*}
\abs{f(x+\delta e_i)-f(x-\delta e_i)-2\dotproduct{\nabla f(x),\delta e_i}}\le L\delta^2,
\end{align*}
which implies in turn the conditions
\begin{align*}
\abs{ \dfrac{1}{2\delta}\brac{f\brac{x+\delta e_i}-f\brac{x-\delta e_i}}-\dfrac{\partial f}{\partial x_i}(x)}\le \dfrac{L\delta}{2}
\end{align*}
for all $i=1,\ldots,n$. Therefore, we get
\begin{align*}
\norm{\mathcal{G}(x,\delta)-\nabla f(x)}&=\sqrt{\sum_{i=1}^n \brac{\dfrac{1}{2\delta}\brac{f(x+\delta  e_i)-f\brac{x-\delta e_i}}-\dfrac{\partial f}{\partial x_i}(x)}^2}\le \dfrac{L\sqrt{n}\delta}{2},
\end{align*}
which brings us to \eqref{general ball error} and thus justifies (i).

Assertion (ii) follows directly from (i). To verify (iii), take some  $\Delta>0$ and a bounded set $\Omega\subset\R^n$, and then find $r>0$ such that $\Omega\subset r\overline{\mathbb{B}}$. Defining $\Theta:=(r+\Delta)\overline{\mathbb{B}}$, it is clear that $\Theta$ is compact. Since $\nabla f$ is locally Lipschitzian, it is Lipschitz continuous on $\Theta$ with some constant $L>0$. Taking any $(x,\delta)\in \Omega\times (0,\Delta],$ we get that $\overline{\mathbb{B}}(x,\delta)\subset \Theta$, and thus $\nabla f$ is Lipschitz continuous on $\overline{\mathbb{B}}(x,\delta)$ with the constant $L$. Employing finally (i), assertion (iii) is justified.
\end{proof}
The following example shows that when the local Lipschitz continuity of $\nabla f$ is replaced by merely the continuity of $\nabla f$, the finite differences may not be local approximation of $\nabla f$.
\begin{Example}\rm
Define the univariate real-valued function $f$ by
\begin{align*}
f(x):=\begin{cases}
\frac{2}{3}\sqrt{x^3}&\text{ if }x\ge 0,\\
-\frac{2}{3}\sqrt{-x^3}&\text{ if }x<0.\\
\end{cases}
\end{align*}
The derivative of $f$ is calculated by
\begin{align*}
\nabla f(x)=\begin{cases}
\sqrt{x}&\text{ if }x\ge 0,\\
\sqrt{-x}&\text{ if }x< 0\\
\end{cases}
\end{align*}
being clearly continuous on $\R$ while not Lipschitz continuous around $0$. If we suppose that $\mathcal{G}(x,\delta)$ is the forward finite difference approximation of $\nabla f(x)$ from \eqref{forward}, we get that
\begin{align*}
\mathcal{G}(0,\delta)=\dfrac{f(\delta)-f(0)}{\delta}=\dfrac{ \frac{2}{3}\sqrt{\delta^3}}{\delta}=\dfrac{2\sqrt{\delta}}{3}\; \text{ for all }\;\delta>0,
\end{align*}
which implies that $\mathcal{G}(0,\delta)/\delta\rightarrow\infty$ as $\delta\downarrow 0$. It follows from \eqref{limsup bound} that $\mathcal{G}(x,\delta)$ is not a local approximation of the derivative $\nabla f$. Supposing now that $\mathcal{G}(x,\delta)$ is the central finite difference approximation of $\nabla f(x)$, we deduce from \eqref{central} the expression
 \begin{align*}
\mathcal{G}(0,\delta)=\dfrac{f(\delta)-f(-\delta)}{2\delta}=\dfrac{4\sqrt{\delta}}{3}\;\text{ for all }\;\delta>0,
 \end{align*}
 which also tells us that $\mathcal{G}(x,\delta)$ is not a local approximation of $\nabla f$.
 \end{Example}\vspace*{-0.1in}

 \section{General Derivative-Free Methods for \texorpdfstring{$\mathcal{C}^{1,1}_L$}{Lg} Functions}\label{sec:4}
 
This section addresses the optimization problem \eqref{optim prob} when $f\in \mathcal{C}^{1,1}_L$ for some $L>0$. In the context of derivative-free optimization, we assume that the only information about the function value $f(x)$ is available. This means that $\nabla f(x)$ is not exactly computable at every $x\in\R^n$, and the Lipschitz constant $L$ of $\nabla f$ is also not available. By employing gradient approximation methods that satisfy the global error bound \eqref{global approx}, we propose here the general {\em Derivative-Free method with Constant stepsize} (DFC) to solve this problem with providing its convergence analysis. The DFC algorithm is described as follows.

\begin{longfbox}
\begin{Algorithm}[{DFC}]\hlabel{GDF}
\setcounter{Step}{-1}
\begin{Step}\rm Choose a global approximation $\mathcal{G}$ of $\nabla f$ under condition \eqref{global approx}. Select an initial point $x^1\in\R^n,$ an initial sampling radius $\delta_1>0,$ a  constant $C_1>0,$ a reduction factor $\theta\in (0,1)$, and scaling factors $\mu>2,r>1,\kappa>0$. Set $k:=1.$
\end{Step}
\begin{Step}[approximate gradient]\rm 
Find $g^k$ and the smallest nonnegative integer $i_k$ such that
\begin{align*}
g^k=\mathcal{G}(x^k,\theta^{i_k}\delta_k) \text{ and }\norm{g^k}>\mu C_k \theta^{i_k}\delta_k.
\end{align*}
Then set $\delta_{k+1}:=\theta^{i_k}\delta_k.$
\end{Step}
\begin{Step}[update]\rm
If $f\Big(x^{k}-\dfrac{\kappa}{C_k}g^k\Big)\le f(x^k)-\dfrac{\kappa(\mu-2)}{2C_k\mu}\norm{g^k}^2$, then $x^{k+1}:=x^k-\dfrac{\kappa}{C_k}g^k$ and $C_{k+1}:=C_k.$ Otherwise, $x^{k+1}:=x^k$ and $C_{k+1}:=rC_{k}.$
\end{Step}
\end{Algorithm}
\end{longfbox}
\begin{Remark}\rm
Let us now present some observations on Algorithm~\ref{GDF}. The first observation clarifies the existence of $g^k$ and $i_k$ in Step~1. Observation (ii) presents the ideas behind the construction of the algorithm. Observation (iii) explains the iteration updates in Step~2 while observation (iv) interprets the term ``constant stepsize" in the name of our method.

{\bf(i)} The procedure of finding $g^k$ and $i_k$ that satisfy Step~1 can be given as follows. Set $i_k=0$ and set
\begin{align}\label{recal}
g^k=\mathcal{G}(x^k,\delta_k).
\end{align}
While $\norm{g^k}\le \mu C_k\theta^{i_k}\delta_k$, increase $i_k$ by $1$ and recalculate $g^k$ under \eqref{recal}. When $\nabla f(x^k)\ne 0$, the existence of $g^k$ and $i_k$ in Step~1 is guaranteed. Indeed, otherwise we get a sequence $\set{g^k_i}$ with
\begin{align}\label{two ineq gki}
g^{k}_i=\mathcal{G}(x^k,\theta^i\delta_k)\text{ and }\;\norm{g^k_i}\le \mu \theta^i\delta_k\;\text{ for all }\;i\in\N.
\end{align}
Since $\theta\in (0,1),$ the latter means that $g^k_i\rightarrow0$ as $i\rightarrow\infty.$
Since $\mathcal{G}$ is a global approximation of $\nabla f$, for $C>0$ given in \eqref{global approx}, we get 
\begin{align*}
\norm{g^k_i-\nabla f(x^k)}\le C\theta^i\delta_k\text{ for all }i\in\N. 
\end{align*}
Letting $i\rightarrow\infty$ with taking into account that $g^k_i\rightarrow0$, the latter inequality implies  that $\nabla f(x^k)=0$, which is a contradiction. 

{\bf(ii)} Since $\mathcal{G}$ is a global approximation of $\nabla f$, let $C$ be a constant given in \eqref{global approx}.  Since $C$ is unknown, Algorithm \ref{GDF} generates the sequence $\set{C_k}$ that approximates $C$. Assume that in some iteration $k^{\rm th},$ the constant $C_k$ is slightly larger than $C$. Then Step~1 of Algorithm~\ref{GDF} provides
\begin{align}\label{condition remark}
\norm{g^k-\nabla f(x^k)}\le C\delta_{k+1}\le C_k\delta_{k+1}< \dfrac{1}{\mu}\norm{g^k}.
\end{align}
Since $\mu>2$, the inequality above gives us $ \norm{g^k-\nabla f(x^k)}\le\dfrac{1}{\mu}\norm{g^k},$ which yields the angle condition $\dotproduct{g^k,\nabla f(x^k)}>0$ as illustrated in the Figure \ref{fig:1} below.
\begin{figure}[H]
\centering
\includegraphics[scale=0.25]{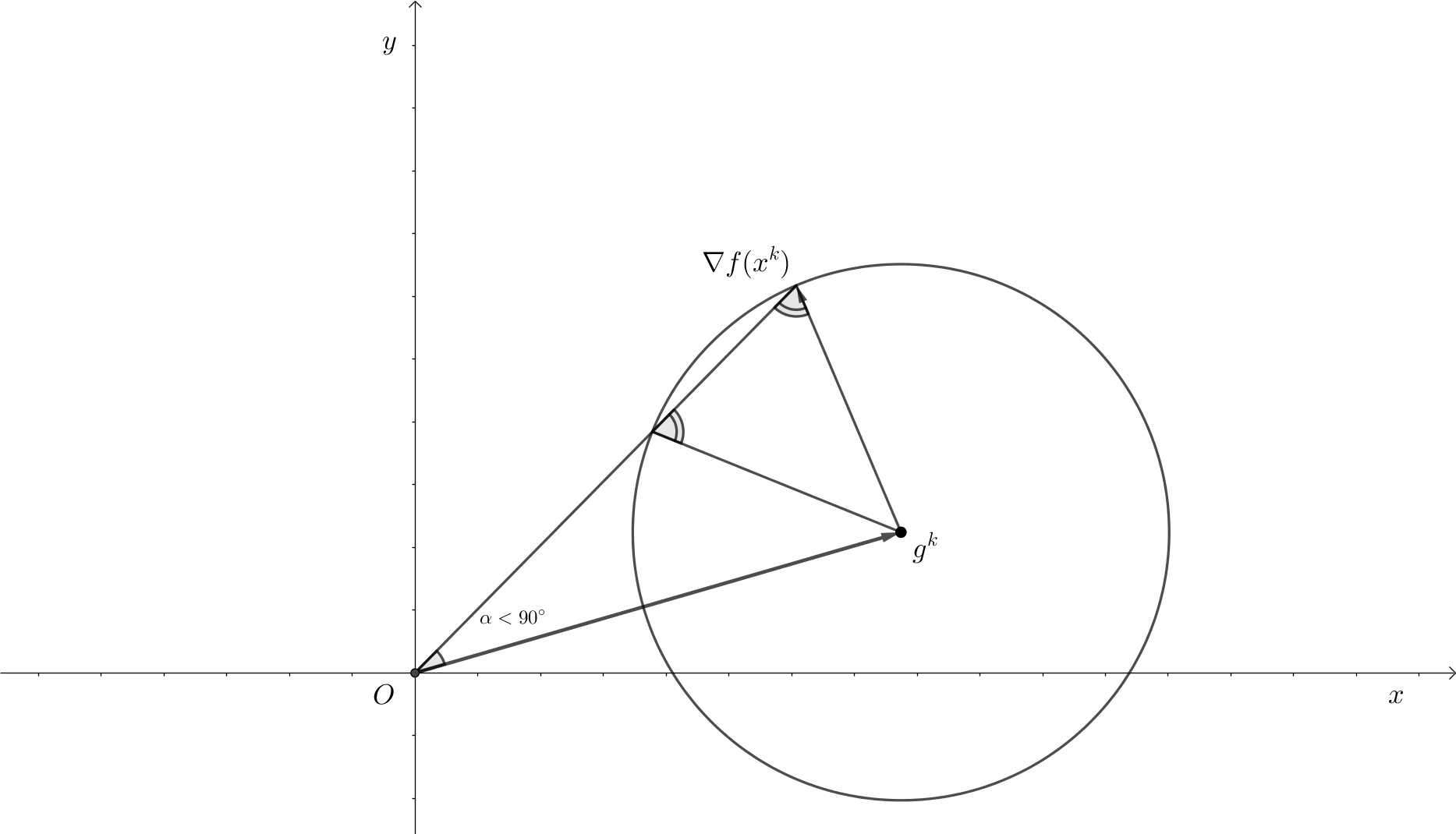}
\caption{An illustration for the angle condition}
\label{fig:1}
\end{figure}
\noindent
Moreover, the result of \cite[Theorem~3.2]{kmt23.1} about the convergence of general inexact gradient methods also tells us that the sequence $\set{x^k}$ satisfying the procedure $x^{k+1}=x^k-\frac{1}{L}g^k$ and condition \eqref{condition remark} achieves fundamental convergence properties. Although the Lipschitz constant $L$ of $\nabla f$ \textit{is not available} in our context, it is linearly related to $C$ for some choice of $\mathcal{G}$; e.g., when $\mathcal{G}$ is obtained from finite differences, we get by Proposition~\ref{for err bound} that $C=\frac{L\sqrt{n}}{2}$. Thus the procedure $x^{k+1}=x^k-\frac{1}{L}g^k$ can be replaced by
\begin{align*}
x^{k+1}&=x^{k}-C^{-1}\kappa g^k\approx x^{k}-C_k^{-1}\kappa g^k \;\text{ for some constant }\;\kappa>0.
\end{align*}
In the formulation of Algorithm~\ref{GDF}, the constant $\kappa$ can be chosen as an arbitrary positive number, which creates more flexibility for our method. 

{\bf(iii)} The condition  $f(x^{k}-C_k^{-1}\kappa g^k)\le f(x^k)-\dfrac{\kappa(\mu-2)}{2C_k\mu}\norm{g^k}^2$ determines whether $C_k$ is a good approximation for $C$ in the sense that the objective function $f$ is sufficiently decreasing when the iterate moves from $x^k$ to $x^{k+1}:=x^k-C_k^{-1}\kappa g^k$. If the condition does not hold, we increase $C_k$ by setting $C_{k+1}:= rC_k$ to get a better approximation for $C$ and stagnate the iterative sequence by setting $x^{k+1}:=x^k$.

{\bf(iv)} It will be shown in Proposition~\ref{claim 1 sec 5} that there exists a positive number $\bar C>0$ such that $C_k=\bar C$ for sufficiently large $k\in\N$, which also implies that $x^{k+1}=x^k-\kappa\bar C^{-1}g^k$ for such $k$. This explains the term ``constant stepsize'' in the name of our algorithm. 
\end{Remark}
The next proposition verifies that the tail of the sequence $\set{C_k}$ generated by Algorithm~\ref{GDF} is constant, which is crucial in deriving the fundamental convergence properties of the algorithm. 
\begin{Proposition}\label{claim 1 sec 5}
Let $\set{C_k}$ be the sequence generated by Algorithm~{\rm\ref{GDF}}. Assume that $\nabla f(x^k)\ne 0$ for all $k\in\N$. Then there exists a number $N\in\N$ such that $C_{k+1}=C_k$ whenever $k\ge N$.
\end{Proposition}
\begin{proof}
Since $\mathcal{G}$ is a global approximation of $\nabla f$ under condition \eqref{global approx}, there exists $C>0$ such that 
\begin{align}\label{global bound proof 1}
\norm{\mathcal{G}(x,\delta)-\nabla f(x)}\le C\delta\;\text{ for all }\;(x,\delta)\in\R^n\times (0,\infty).
\end{align}
Since $\nabla f$ is globally Lipschitz continuous, we also find $L>0$ such that $\nabla f$ is Lipschitz continuous with the constant $L$ on $\R^n$. Arguing by contradiction, suppose that the number $N$ asserted in the proposition does not exist. By Step~2 of Algorithm~\ref{GDF}, this implies that $C_{k+1}=rC_k$ for infinitely many $k\in\N$, and hence $C_k\rightarrow\infty$. Therefore, there is a number $K\in\N$ such that $C_{K+1}=r C_K$ and $C_K>\max\set{C,L\kappa}$. Using Step~2 of Algorithm~\ref{GDF} together with the update $C_{K+1}=r C_K$, we deduce that
\begin{align}\label{less}
f\Big(x^{K}-\dfrac{\kappa}{C_K}g^{K}\Big)> f(x^{K})-\dfrac{\kappa(\mu-2)}{2C_K\mu}\norm{g^{K}}^2.
\end{align}
Combining $g^K=\mathcal{G}(x^K,\delta_{K+1})$ and $\norm{g^K}\ge \mu C_K\delta_{K+1}$ from Step~1 of Algorithm~\ref{GDF} with \eqref{global bound proof 1} and $C_K>C$ as above, we get the relationships 
\begin{align*}
\norm{g^{K}-\nabla f(x^{K})}&=\norm{\mathcal{G}(x^{K},\delta_{K+1})-\nabla f(x^{K})}\\
&\le C\delta_{K+1}\le C_K\delta_{K+1}\le\mu^{-1}\norm{g^{K}}.
\end{align*}
By the Cauchy-Schwarz inequality, the latter gives us
\begin{align*}
\dotproduct{\nabla f(x^{K}),g^{K}}&=\dotproduct{\nabla f(x^{K})-g^{K},g^{K}}+\norm{g^{K}}^2\\
&\ge -\norm{\nabla f(x^{K})-g^{K}}\norm{g^{K}}+\norm{g^{K}}^2\\
&\ge(1-\mu^{-1})\norm{g^{K}}^2.
\end{align*}
Combining this with Lemma \ref{lemma descent} and taking into account the global Lipschitz continuity of $\nabla f$ with the constant $L$ as well as the condition $C_K>L\kappa$ as above, we get that 
\begin{align*}
f\Big(x^{K}-\dfrac{\kappa}{C_K}g^{K}\Big)-f(x^{K})&\le -\dfrac{\kappa}{C_K}\dotproduct{\nabla f(x^{K}),g^{K}}+\dfrac{L}{2}\norm{\dfrac{\kappa}{C_K}g^{K}}^2\\
&\le -\dfrac{\kappa}{C_K}\Big(1-\dfrac{1}{\mu}\Big)\norm{g^{K}}^2+\dfrac{\kappa}{2C_K}\norm{g^{K}}^2\\
&=-\dfrac{\kappa}{C_K}\norm{g^{K}}^2\Big(\dfrac{1}{2}-\dfrac{1}{\mu}\Big)=-\dfrac{\kappa(\mu-2)}{2C_K\mu}\norm{g^{K}}^2,
\end{align*}
which clearly contradicts\eqref{less} and completes the proof of the proposition.
\end{proof}

Now we are ready to establish the convergence properties of Algorithm~\ref{GDF}.
\begin{Theorem}\label{convergence GDF}
Let $\set{x^k}$ be the sequence generated by Algorithm~{\rm\ref{GDF}} and assume that $\nabla f(x^k)\ne0$ for all $k\in\N.$ Then either $f(x^k)\rightarrow-\infty$, or we have the assertions:

{\bf(i)} The sequence $\set{\nabla f(x^k)}$ converges to $0$.

{\bf(ii)} If $f$ satisfies the KL property at some accumulation point $\bar x$ of $\set{x^k}$, then $x^k\rightarrow\bar x$.

{\bf(iii)} If $f$ satisfies the KL property at some accumulation point $\bar x$ of $\set{x^k}$ with $\psi(t)=Mt^{q}$ for $M>0$ and $q\in (0,1)$, then the following convergence rates are guaranteed for $\set{x^k}$:

$\bullet$ If $q\in(0,1/2]$, then $\set{x^k}$, $\set{\nabla f(x^k)}$, and $\set{f(x^k)}$ converge linearly to $\bar x$, $0$, and $f(\bar x)$, respectively.

$\bullet$ The setting of $q\in(1/2,1)$ ensures the estimates
$$
\norm{x^k-\bar x}=\mathcal{O}\Big(k^{-\frac{1-q}{2q-1}}\Big),\; \norm{\nabla f(x^k)}=\mathcal{O}\Big(k^{-\frac{1-q}{2q-1}}\Big),\;\text{ and }\;f(x^k)-f(\bar x)=\mathcal{O}\Big(k^{-\frac{2-2q}{2q-1}}\Big).
$$
\end{Theorem}
\begin{proof}
Since $\mathcal{G}$ is a global approximation of $\nabla f$ under condition \eqref{global approx}, there is $C>0$ such that 
\begin{align}\label{global bound proof}
\norm{\mathcal{G}(x,\delta)-\nabla f(x)}\le C\delta\;\text{ for all }\;(x,\delta)\in \R^n\times (0,\infty).
\end{align}
By $f\in{\cal C}^{1,1}_L$, we find $L>0$ such that the gradient mapping $\nabla f$ is Lipschitz continuous with the constant $L$ on $\R^n$. Taking the number $N\in\N$ from Proposition~\ref{claim 1 sec 5} ensures that $C_k=C_N$ for all $k\ge N$. This implies by Step~2 of Algorithm~\ref{GDF} that
\begin{align}\label{ine leq DFO}
f(x^{k+1})= f\Big(x^k-\dfrac{\kappa}{C_N}g^k\Big)\le f(x^k)-\dfrac{\kappa(\mu-2)}{2C_N\mu}\norm{g^k}^2\;\text{ for all }\;k\ge N,
\end{align}
which tells us that $\set{f(x^k)}_{k\ge N}$ is decreasing. If $f(x^k)\rightarrow-\infty$, there is nothing to prove, so we assume that $f(x^k)\not\rightarrow-\infty$, which implies that $\set{f(x^k)}$ is convergent. As a consequence, we get $f(x^k)-f(x^{k+1})\rightarrow 0$ as $k\to\infty$. Then \eqref{ine leq DFO} tells us that $g^k\rightarrow 0$. From Step~1 of Algorithm~\ref{GDF} it follows that
\begin{align}\label{5.3}
\norm{g^k}>\mu C_k\delta_{k+1}=\mu C_N\delta_{k+1}\;\text{ for all }\;k\ge N,
\end{align}
which implies that $\delta_{k+1}\downarrow 0$ as $k\rightarrow\infty$. It further follows from $g^k=\mathcal{G}(x^k,\delta_{k+1})$ and \eqref{global bound proof} that
\begin{align}\label{5.4}
\norm{g^k-\nabla f(x^k)}=\norm{\mathcal{G} (x^k,\delta_{k+1})-\nabla f(x^k)}\le C \delta_{k+1} \text{ for all }k\in\N,
\end{align}
which yields $\nabla f(x^k)\rightarrow 0$ as $k\rightarrow\infty$ and thus justifies (i).

To verify (ii), take any accumulation point $\bar x$ of $\set{x^k}$ and assume that $f$ satisfies the KL property at $\bar x$. By \eqref{5.3} and \eqref{5.4}, we obtain that
\begin{equation}\label{2 ineq}
\begin{array}{ll}
\norm{\nabla f(x^k)}&\le\norm{g^k}+\norm{\nabla f(x^k)-g^k}\le \norm{g^k}+C\delta_{k+1} \nonumber\\
&\le \norm{g^k}+\dfrac{C\norm{g^k}}{\mu C_N} =\alpha\norm{g^k}\text{ for all }k\ge N,
\end{array}
\end{equation}
where $\alpha:=\frac{\mu C_N+C}{\mu C_N}$. This together with \eqref{ine leq DFO} yields condition \eqref{two conditions} in Proposition~\ref{general rate}. By Remark~\ref{rmk two KL}(i), assumptions (H1) and (H2) in Proposition~\ref{general convergence under KL} hold. Therefore, $x^k\rightarrow\bar x$ as $k\to\infty$, which justifies (ii).

To proceed with the proof of assertion (iii) under the KL property at $\ox$ with $\psi(t)=Mt^{q}$, we use the iterations $x^{k+1}=x^k-C^{-1}\kappa g^k$ as in Step~2 of Algorithm~\ref{GDF} together with $\|g^k\|>0$ from Step~1 of Algorithm~\ref{GDF}. This gives us $x^{k+1}\ne x^k$ for $k\ge N$. Combining the latter with \eqref{ine leq DFO} and \eqref{2 ineq}, we see that all the assumptions in Proposition~\ref{general rate} are satisfied. This verifies the convergence rates of $\set{x^k}$ to $\bar x$ stated in (iii). Since $\bar x$ is an accumulation point of $\set{x^k}$, it follows from (i) that $\bar x$ is a stationary point of $f$, i.e., $\nabla f(\bar x)=0$. Hence the usage of Lemma~\ref{lemma descent} and the decreasing property of $\set{f(x^k)}_{k\ge N}$ yields
\begin{align*}
0\le f(x^k)-f(\bar x)\le \dotproduct{\nabla f(\bar x),x^k-\bar x}+\dfrac{{L}}{2}\norm{x^k-\bar x}^2=\dfrac{{L}}{2}\norm{x^k-\bar x}^2,
\end{align*}
which justifies the convergence rates of $\set{f(x^k)}$ to $f(\bar x)$ as asserted in (iii).

It remains to verify the convergence rates for $\set{\nabla f(x^k)}$. Since $\nabla f$ is Lipschitz continuous with the constant $L>0$, the claimed property follows from the convergence rates for $\set{x^k}$ due to
\begin{align*}
\norm{\nabla f(x^k)}=\norm{\nabla f(x^k)-\nabla f(\bar x)}\le L\norm{x^k-\bar x}.
\end{align*}
This therefore completes the proof of the theorem.
\end{proof}\vspace*{-0.2in}
\section{General Derivative-free Methods for \texorpdfstring{$\mathcal{C}^{1,1}$}{Lg} Functions}\label{sec:5}
In this section, we propose and justify the novel {\em Derivative-Free method with Backtracking stepsize} (DFB) to solve the optimization problem \eqref{optim prob}. The main result here establishes the {\em global convergence} with convergence rates of the following algorithm, which employs gradient approximations satisfying the local error bound estimate \eqref{local approx}.

\begin{longfbox}
 \begin{Algorithm}[DFB]\hlabel{GDF C11}
\setcounter{Step}{-1}
\begin{Step}[initialization]\rm Choose a local approximation $\mathcal{G}$ of $\nabla f$ under condition \eqref{local approx}. Select an initial point $x^1\in\R^n$ and initial radius $\delta_1>0$, a constant $C_1>0$, factors $\theta\in(0,1)$, $\mu>2,\eta>1$, linesearch constants $\beta\in(0,1/2),\gamma\in(0,1)$, $\bar\tau>0$, and an initial bound $ t^{\min}_1\in(0,\bar \tau)$. Choose a sequence of manually controlled errors $\set{\nu_k}\subset[0,\infty)$ such that $\nu_k\downarrow 0$ as $k\to\infty$. Set $k:=1.$
\end{Step}
\begin{Step}[approximate gradient]\rm 
Select $g^k$ and the smallest nonnegative integer $i_k$ so that
\begin{align}\label{choice gk}
g^k=\mathcal{G}(x^k,\min\set{\theta^{i_k}\delta_k,\nu_k}) \text{ and }\norm{g^k}>\mu C_k\theta^{i_k}\delta_k.
\end{align}
Then set $\delta_{k+1}:=\theta^{i_k}\theta_k$.
\end{Step}
\begin{Step}[linesearch]\rm
Set $t_k:=\bar \tau$. While 
\begin{align}\label{GDF linesearch}
f(x^k-t_kg^k)> f(x^k)-\beta t_k\norm{g^k}^2\text{ and }t_k\ge t_k^{\min},
\end{align}
set $t_k:=\gamma t_k.$
\end{Step}
\begin{Step}[stepsize and parameters update]\rm
If $t_k\ge  t^{\min}_{k}$, then set $\tau_k:=t_k$, $C_{k+1}:=C_k$, and $t^{\min}_{k+1}:= t^{\min}_{k}$. Otherwise, set $\tau_k:=0,\;C_{k+1}:=\eta C_k$ and $t^{\min}_{k+1}:=\gamma t^{\min}_{k}$. 
\end{Step}
\begin{Step}[iteration update]\rm Set $x^{k+1}:=x^k-\tau_kg^k.$ Increase $k$ by $1$ and go back to Step 1.
\end{Step}
\end{Algorithm}
\end{longfbox}
\begin{Remark}\rm\label{Remark C11}

{\bf(i)} Fix any $k\in\N$. The procedure of finding $g^k$ and $i_k$ that satisfies Step~1 of Algorithm~\ref{GDF C11} can be given as follows. Set $i_k:=0$ and calculate $g^k$ as 
\begin{align}\label{GDF recal}
g^k=\mathcal{G}\big(x^k,\min\{\theta^{i_k}\delta_k,\nu_k\}\big).
\end{align}
While $\norm{g^k}\le \mu C_k\theta^{i_k}\delta_k $, increase $i_k$ by $1$ and recalculate $g^k$ by formula \eqref{GDF recal}. We now show that when $\nabla f(x^k)\ne 0$, this procedure stops after a finite number of steps giving us $g^k$ and $i_k$ as desired. Indeed, assuming the contrary that the procedure does not stop, we get a sequence of $\set{g^k_i}$ with
\begin{align}\label{gki select}
g^k_i=\mathcal{G}\big(x^k,\min\set{\theta^i\delta_k,\nu_k}\big)\;\text{ and }\;\norm{g^k_i}\le \mu C_k\theta^i\delta_k\;\text{ for all }\;i\in\N.
\end{align}
Since $\mathcal{G}$ is a local approximation of $\nabla f$, for any fixed $\Delta>0$ the condition \eqref{local approx} with $\Omega=\set{x^k}$ ensures the existence of a positive number $C$ such that 
\begin{align}\label{bound in remark}
\norm{\mathcal{G}(x^k,\delta)-\nabla f(x^k)}\le C\delta\;\text{ whenever }\;0<\delta\le\Delta.
\end{align}
By $\theta\in(0,1)$, there is $N\in\N$ with $\theta^i\delta_k\le \Delta$ for all $i\ge N$. Combining this with  \eqref{gki select} and \eqref{bound in remark} yields
\begin{align*}
\norm{g^k_i-\nabla f(x^k)}\le C\theta^i\delta_k \;\text{ and }\;  \norm{g^k_i}\le\mu C_k\theta^i\delta_k\;\text{ for all }\;i\ge N.
\end{align*}
Letting $i\rightarrow\infty$, we arrive at $\nabla f(x^k)=0$, which is a contradiction.

{\bf(ii)} It follows directly from the construction of $\delta_k$ in Step~1 of Algorithm~\ref{GDF C11} that
\begin{align}\label{short rela}
g^k=\mathcal{G}(x^k,\min\set{\delta_{k+1},\nu_k})  \text{ and }\norm{g^k}>\mu C_k\delta_{k+1}.
\end{align}
\end{Remark}

To proceed further with the convergence analysis of Algorithm~\ref{GDF C11}, we obtain two results of their independent interest. The first one reveals some uniformity of general linesearch procedures with respect to the selections of reference points, stepsizes, and directions.
\begin{Lemma}\label{lemm t uniform}
Let $f:\R^n\rightarrow\R$ be a function with a locally Lipschitz continuous gradient, and let $\beta\in(0,1/2)$. Then for any nonempty bounded set $\Omega\subset\R^n$, there exists $\bar t>0$ such that 
\begin{align*}
f(x-tg)\le f(x)-\beta t\norm{g}^2\;\text{ whenever }\;x\in\Omega,\;2\norm{g-\nabla f(x)}\le\norm{g},\;\text{ and }\;t\in(0,\bar t].
\end{align*}
\end{Lemma}
\begin{proof}
The boundedness of $\Omega$ gives us $r>0$ such that $\Omega\subset r\overline{\mathbb{B}}$. Using the continuity of $\nabla f$ and the compactness of $r\overline{\mathbb{B}}$, we define $r':=\max \set{\norm{\nabla f(x)}\;|\;x\in r\overline{\mathbb{B}}}$.
Since $f\in\mathcal{C}^{1,1}$, there is $L>0$ such that $\nabla f$ is Lipschitz continuous with the constant $L$ on $\Theta:=(r+2r')\overline{\mathbb{B}}$. By $\beta<1/2$, we find $\bar t>0$ with $\bar t<\min \set{1,L^{-1}(1-2\beta)}$. Now take some $x\in\Omega\subset\Theta$ and $g\in\R^n$ such that $2\norm{g-\nabla f(x)}\le\norm{g}$ and $t\in (0,\bar t]$. The choice of $g$ gives us by the Cauchy-Schwarz inequality that
\begin{equation}\label{dot inequa}
\begin{array}{ll}
\dotproduct{\nabla f(x),g}&= \dotproduct{\nabla f(x)-g,g}+\norm{g}^2\ge -\norm{\nabla f(x)-g}\norm{g}+\norm{g}^2\\
&\ge-\frac{1}{2}\norm{g}^2+\norm{g}^2=\frac{1}{2}\norm{g}^2,
\end{array}
\end{equation}
and by using the triangle inequality that
\begin{eqnarray*}
\begin{array}{ll}
\norm{\nabla f(x)}\ge\norm{g}-\norm{g-\nabla f(x)}\ge\frac{1}{2}\norm{g}.
\end{array}
\end{eqnarray*}
Combining the latter with the choice of $t,\;\bar t$, $x\in\Omega\subset r\overline{\mathbb{B}}$ and the construction of $r'$ yields
\begin{align*}
t\norm{g}\le \bar t\norm{g}\le2\bar t \norm{\nabla f(x)}\le 2\bar t r'<2r',
\end{align*}
which ensures that $x-tg\in\Theta$. The convexity of $\Theta$ tells us that the entire line segment $[x,x-tg]$ lies on $\Theta$. Remembering that $\nabla f$ is Lipschitz continuous with the constant $L$ on $\Theta$, we employ Lemma~\ref{lemma descent} by taking into account that $t\le\bar t<L^{-1}(1-2\beta)$ and that \eqref{dot inequa}. This gives us
\begin{align*}
f(x-tg)-f(x)&\le \dotproduct{x-tg-x,\nabla f(x)}+\dfrac{L}{2}\norm{x-tg-x}^2\\
&=-t\dotproduct{g,\nabla f(x)}+\dfrac{Lt^2}{2}\norm{g}^2\le -\frac{t}{2}\norm{g}^2+\dfrac{Lt^2}{2}\norm{g}^2\\
&=-\beta t\norm{g}^2+t\norm{g}^2\dfrac{2\beta -1+Lt}{2}\le -\beta t\norm{g}^2
\end{align*}
and thus completes the proof of the lemma.
\end{proof}

Employing the obtained lemma, we derive the next result showing that unless the stationary point is found, Algorithm~\ref{GDF C11} always makes a progress after a finite number of iterations.
\begin{Proposition}\label{prop finite} Let $\set{x^k}$ and $\set{\tau_k}$ be the sequences generated by Algorithm~{\rm\ref{GDF C11}}, and let $K\in\N$ be such that $\nabla f(x^K)\ne 0$. Then we can choose a number $N\ge K$ so that $\tau_N>0$.
\end{Proposition}
\begin{proof}
Assume on the contrary that $\tau_k=0$ for all $k\ge K$. Steps~3 and 4 of Algorithm~\ref{GDF C11} give us
\begin{align}\label{Lk increasing}
t^{\min}_{k+1}=\gamma t^{\min}_k\;\text{ and }\;x^k=x^K\;\text{ for all }\;k\ge K.
\end{align}
Therefore, $\nabla f(x^k)=\nabla f(x^K)\ne 0$ for all $k\ge K,$ which implies that $g^k$ and $i_k$ in Step 1 of Algorithm~\ref{GDF C11}  exist for all $k\ge K$. Since $\mathcal{G}$ is a local approximation of $\nabla f$, for any fixed $\Delta>0$ condition \eqref{local approx} with $\Omega=\set{x^K}$ ensures the existence of  $C>0$ with
\begin{align}\label{bound for G(x,delta)}
\norm{\mathcal{G}(x^K,\delta)-\nabla f(x^K)}\le C \delta\;\text{ whenever }\;0<\delta\le\Delta.
\end{align}
It follows from Lemma~\ref{lemm t uniform} with $\Omega=\set{x^K}$ that there is some $\bar t>0$  such that
\begin{align}\label{choi bar t}
f(x^K-tg)\le f(x^K)-\beta t\norm{g}^2\;\text{ whenever }\;
2\norm{g-\nabla f(x^K)}\le \norm{g}\;\text{ and }\;t\in (0,\bar t].
\end{align}
Using $\nu_k\downarrow0,\; t_k^{\min}\downarrow0$, $\nabla f(x^K)\ne 0$, and \eqref{Lk increasing}  gives us $N\ge K$ for which $\nu_N<\min\set{\Delta,\frac{1}{3C}\norm{\nabla f(x^K)}}$ and $ t^{\min}_N<\gamma\bar t$. Then we get from \eqref{bound for G(x,delta)} with taking into account $x^N=x^K$ that
\begin{align*}
\norm{\mathcal{G}(x^{N},\min\set{\delta_{N+1},\nu_N})-\nabla f(x^{N})}\le C\min\set{\delta_{N+1},\nu_N}\le C\nu_N\le \dfrac{1}{3}\norm{\nabla f(x^N)}.
\end{align*}
Combining this with $g^N=\mathcal{G}(x^{N},\min\set{\delta_{N+1},\nu_N})$ from \eqref{short rela} provides the estimate
\begin{align*}
\norm{g^{N}-\nabla f(x^N)}\le \dfrac{1}{3}\norm{\nabla f(x^N)}
\end{align*}
which implies by the triangle inequality that
\begin{align*}
\norm{g^N}&\ge \norm{\nabla f(x^N)}-\norm{g^N-\nabla f(x^N)}\ge 2 \norm{g^N-\nabla f(x^N)}.
\end{align*}
Employing the latter together  with \eqref{choi bar t} and $x^N=x^K$ yields
\begin{align}\label{uniform con}
f(x^N-tg^N)\le f(x^N)-\beta t\norm{g^N}^2 \text{ for all }t\in (0,\bar t].
\end{align}
It follows from \eqref{uniform con} and the choice of parameters that
\begin{align*}
\max\big\{t\;\big|\;f(x^N-t g^N)\le f(x^N)-\beta t\norm{g^N}^2,\;t=\bar \tau,\bar \tau\gamma,\bar \tau\gamma^2,\ldots\big\}>\gamma\bar t>t_N^{\min},
\end{align*}
which implies in turn by Step~2 of Algorithm~\ref{GDF C11} that
\begin{align*}
t_N=\max\big\{t\;|\;f(x^N-t g ^N)\le f(x^N)-\beta t\norm{g^N}^2,\;t=\bar \tau,\bar\tau\gamma,\bar \tau\gamma^2,\ldots \big\}>t_N^{\min}.
\end{align*}
By Step~3 of Algorithm~\ref{GDF C11}, we conclude that $\tau_N=t_N>0$, a contradiction which completes the proof.
\end{proof}
Now we are ready the establish the convergence properties of Algorithm~\ref{GDF C11}.
\begin{Theorem}\label{theo c11}
Let $\set{x^k}$ be the sequence generated by Algorithm~{\rm\ref{GDF C11}} and assume that $\nabla f(x^k)\ne 0$ for all $k\in\N.$ Then either $f(x^k)\rightarrow-\infty$, or the following assertions hold:

{\bf(i)} Every accumulation point of $\set{x^k}$ is a stationary point of $f$.

{\bf(ii)} If the sequence $\set{x^k}$ is bounded, then the set of accumulation points of $\set{x^k}$ is nonempty, compact, and connected in $\R^n$.

{\bf(iii)} If $\set{x^k}$ has an isolated accumulation point, then this sequence converges to it.
\end{Theorem}
\begin{proof}
 First it follows from Steps~2 and 3 of Algorithm~\ref{GDF C11} that 
\begin{align}\label{linesearch ineq}
\beta \tau_k\norm{g^k}^2\le f(x^k)-f(x^{k+1})\;\text{ for all }\;k\in\N,
\end{align}
which tells us that $\set{f(x^k)}$ is nonincreasing. If $f(x^k)\rightarrow-\infty$, there is nothing to prove; so we assume that $f(x^k)\not \rightarrow-\infty$, which implies by the nonincreasing property of $\set{f(x^k)}$ that $\inf f(x^k)>-\infty$. Summing up the inequalities in \eqref{linesearch ineq} over $k=1,2,\ldots$ with taking into account that $x^{k+1}=x^k-\tau_kg^k$ from the update in Step~3 of Algorithm~\ref{GDF C11} gives us
\begin{align}\label{two series}
\sum_{k=1}^\infty \tau_k\norm{g^k}^2<\infty\;\text{ and }\; \sum_{k=1}^\infty \norm{g^k}\norm{x^{k+1}-x^k}<\infty.
\end{align}

We divide the proof of (i) into two parts by showing first that the origin is an accumulation point of $\set{g^k}$ and then employing Lemma~\ref{stationary point lemma} to establish the stationarity of all the accumulation points of $\set{x^k}.$
\setcounter{Claim}{0}
\begin{Claim}\label{cl 1} The origin $0\in\R^n$ is an accumulation point of the sequence $\set{g^k}$.
\end{Claim}
\noindent 
Arguing by contradiction, suppose that there are numbers $\varepsilon>0$ and $K\in\N$ such that
\begin{align}\label{gk lowerbound}
\norm{g^k}\ge \varepsilon\;\text{ for all }k\ge K.
\end{align}
Combining this with \eqref{two series} gives us $\tau_k\downarrow0$ and $\sum_{k=1}^\infty\norm{x^{k+1}-x^k}<\infty$. The latter implies that $\set{x^k}$ converges to some $\bar x\in\R^n$. By taking a larger $K$, we can assume that $\tau_k<\bar \tau$ for all $k\ge K$. Let $\mathcal{N}$ be the set of all $k\in\N$ such that $\tau_k>0$. It follows from Proposition~\ref{prop finite} that $\mathcal{N}$ is infinite. Hence we can take any $k\ge K$ with $k\in\mathcal{N}$ and get that $\tau_k\in(0,\bar\tau)$. By Step~3 of Algorithm~\ref{GDF C11}, we deduce that $\tau_k=t_k\in[t^{\min}_k,\bar\tau)$. Fixing some such $k$, we get from the exit condition in Step~2 of Algorithm~\ref{GDF C11} that
\begin{align}\label{exit condition}
-\gamma^{-1}\beta \tau_k\norm{g^k}^2<f(x^{k}-\gamma^{-1}\tau_kg^k)-f(x^k).
\end{align}
The classical mean value theorem gives us $\tilde{x}^k\in[x^k,x^k-\gamma^{-1}\tau_kg^k]$ such that
\begin{align}\label{def xtilde}
f(x^{k}-\gamma^{-1}\tau_kg^k)-f(x^k)=-\gamma^{-1}\tau_k\dotproduct{g^k,\nabla f(\tilde{x}^k)}.
\end{align}
Combining this with \eqref{exit condition} yields
\begin{align*}
-\gamma^{-1}\beta \tau_k\norm{g^k}^2<-\gamma^{-1}\tau_k\dotproduct{g^k,\nabla f(\tilde{x}^k)},
\end{align*}
which implies by dividing both sides of the inequality by $-\gamma^{-1}\tau_k<0$ that
\begin{align}\label{dot ineq gk}
\dotproduct{g^k,\nabla f(\tilde{x}^k)}<\beta\norm{g^k}^2\;\text{ for all }\;k\ge K,\;k\in\mathcal{N}.
\end{align}
Take some neighborhood $\Omega$ of $\bar x$ and $\Delta>0$. Since $\mathcal{G}$ is a local approximation of $\nabla f$ under condition \eqref{local approx}, there is $C>0$ such that
\begin{align}\label{bound 1}
\norm{\mathcal{G}(x,\delta)-\nabla f(x)}\le C\delta\;\text{ whenever }\;0<\delta\le \Delta\;\text{ and }\;x\in\Omega.
\end{align}
Since $\nu_k\downarrow 0$ and $x^k\rightarrow\bar x$, by taking a larger $K$ we can assume that $\nu_k<\Delta$ and $x^k\in \Omega$ for all $k\ge K$. Using this together with \eqref{bound 1} and  $g^k=\mathcal{G}(x^k,\min\set{\delta_{k+1},\nu_k})$ in \eqref{short rela}, we get that
\begin{align*}
\norm{g^k-\nabla f(x^k)}=\norm{\mathcal{G}(x^k,\min\set{\delta_{k+1},\nu_k})-\nabla f(x^k)}\le C\min\set{\delta_{k+1},\nu_k} \le C\nu_k.
\end{align*}
Combining the latter with $x^k\rightarrow \bar x$, $\nu_k\downarrow0$ as $k\rightarrow\infty$, and the continuity of $\nabla f$ gives us
\begin{align}\label{gk bounded}
g^k\rightarrow\nabla f(\bar x)\text{ as }k\rightarrow\infty,
\end{align}
which yields $\norm{\nabla f(\bar x)}>0$ by \eqref{gk lowerbound}. It follows from \eqref{gk bounded}, $\tau_k\downarrow0$, $x^k\rightarrow\bar x$, and $\tilde{x}^k\in[x^k,x^k-\gamma^{-1}\tau_kg^k]$ for all $k\ge K,k\in\mathcal{N}$ that $\tilde{x}^k\overset{\mathcal{N}}{\rightarrow}\bar x$. Letting $k\overset{\mathcal{N}}{\rightarrow}\infty$ in \eqref{dot ineq gk} with taking into account the convergence above and \eqref{gk bounded} brings us to the estimate
\begin{align*}
\norm{\nabla f(\bar x)}^2\le \beta\norm{\nabla f(\bar x)}^2.
\end{align*}
This contradicts $\beta<\frac{1}{2}$ and $\norm{\nabla f(\bar x)}>0$. Thus the origin is an accumulation point of $\set{g^k}$ as claimed.
\begin{Claim}Every accumulation point of $\set{x^k}$ is a stationary point of $f$.
\end{Claim}
\noindent Take any accumulation point $\bar x$ of $\set{x^k}$. Using  Claim~\ref{cl 1}, the second inequality in \eqref{two series}, and Lemma~\ref{stationary point lemma} tells us that there is an infinite set $J\subset \N$ such that 
\begin{align*}
x^k\overset{J}{\rightarrow}\bar x\text{ and }g^k\overset{J}{\rightarrow}0.
\end{align*}
Take a neighborhood $\Omega$ of $\bar x$ and $\Delta>0$. Since $\mathcal{G}$ is a local approximation of $\nabla f$ under condition \eqref{local approx}, there exists $C>0$ for which
\begin{align}\label{local in proof}
\norm{\mathcal{G}(x,\delta)-\nabla f(x)}\le C\delta\;\text{ whenever }\;0<\delta\le \Delta\;\text{ and }\;x\in\Omega.
\end{align}
Since $\nu_k\downarrow 0$ and $x^k\overset{J}{\rightarrow}\bar x$, we can select $K\in \N$ so that $\nu_k\le\Delta$ and $x^k\in \Omega$ for all $k\ge K,k\in J$. This ensures together with \eqref{local in proof} that
\begin{align*}
\norm{g^k-\nabla f(x^k)}=\norm{\mathcal{G}(x^k,\min\set{\delta_{k+1},\nu_k})-\nabla f(x^k)}\le C\nu_k\;\text{ for all }\;k\ge K,\;k\in J.
\end{align*}
Employing $g^k\overset{J}{\rightarrow}0$ and $\nu_k\downarrow0$ as above, we deduce that $\nabla f(x^k)\overset{J}{\rightarrow}0$, and hence $\nabla f(\bar x)=0.$ Therefore, $\bar x$ is a stationary point of $f,$ which justifies (i).

Now we verify (ii) and (iii) simultaneously. It follows from \eqref{two series} and $\tau_k\le 1$ for all $k\in\N$ by the choice of $\tau_k$ in Step~3 of Algorithm~\ref{GDF C11} that
\begin{align*}
\sum_{k=1}^\infty\norm{x^{k+1}-x^k}^2=\sum_{k=1}^\infty \tau_k^2\norm{g^k}^2\le\bar\tau\sum_{k=1}^\infty \tau_k\norm{g^k}^2<\infty, 
\end{align*}
which implies that $\norm{x^{k+1}-x^k}\rightarrow 0$. Then both assertions (ii) and (iii) follow from Lemma~\ref{lemma: ostrowski}.
\end{proof} 
The next result establishes the global convergence with convergence rates of the iterates $\set{x^k}$ in Algorithm~\ref{GDF C11} under the KL property and the boundedness of $\set{x^k}$. We have already discussed the KL property in Remark~\ref{algebraic}. The boundedness of $\set{x^k}$ is also a standard assumption that appears in many works on gradient descent methods; see, e.g., \cite[Theorem~4.1]{attouch13}, \cite[Theorem~1]{josz23}, and \cite[Assumption~7]{lee19}.
\begin{Theorem}\label{KL GDF C11}
Let $\set{x^k}$ be the sequence of iterates generated by Algorithm~{\rm\ref{GDF C11}}. Assuming that $\nabla f(x^k)\ne 0$ for all $k\in\N$ and that $\set{x^k}$ is bounded yields the assertions:

{\bf(i)} If $\bar x$ is an accumulation point of $\set{x^k}$ and $f$ satisfies the KL property at $\bar x$, then $x^k\rightarrow\bar x$ as $k\to\infty$.

{\bf(ii)} If in addition to {\rm(i)}, the KL property at $\bar x$ is satisfied with $\psi(t)=Mt^{q}$ for some $M>0,q\in(0,1)$, then the following convergence rates are guaranteed:\\
$\bullet$ If $q\in(0,1/2]$, then the sequence $\set{x^k}$ converges linearly to $\bar x$.\\
$\bullet$ If $q\in(1/2,1)$, then we have the estimate
\begin{align*}
\norm{x^k-\bar x}=\mathcal{O}\big( k^{-\frac{1-q}{2q-1}}\big).
\end{align*}
\end{Theorem}
\begin{proof}
Let $\Omega:=\set{x^k}$, and let $\Delta>0$. Since $\mathcal{G}$ is a local approximation of $\nabla f$ satisfying condition \eqref{local approx}, there exists a positive number $C$ such that
\begin{align}\label{accum bound}
\norm{\mathcal{G}(x,\delta)-\nabla f(x)}\le C\delta\;\text{ whenever }\;x\in\Omega\;\text{ and }\;0<\delta\le \Delta.
\end{align}
Select $K\in\N$ so that $\nu_k<\Delta$ for all $k\ge K$, which implies by \eqref{accum bound} and the choice of $g^k$ in Step~1 of Algorithm~\ref{GDF C11} the relationships
\begin{align}\label{equa norm gk}
\norm{g^k-\nabla f(x^k)}&= \norm{\mathcal{G}(x^k,\min\set{\delta_{k+1},\nu_k})-\nabla f(x^k)}\nonumber\\
&\le C\min\set{\delta_{k+1},\nu_k}\le C\delta_{k+1}\;\text{ for all }\;k\ge K.
\end{align}
We split the proof of the result into two parts by showing first that the sequences $\set{C_k}$ and $\set{t_k^{\min}}$ are constant after a finite number of iterations and verifying then the convergence of $\set{x^k}$ in (i) with the rates in (ii) by using Propositions~\ref{general convergence under KL} and \ref{general rate}.
\setcounter{Claim}{0}
\begin{Claim}\label{cl 1 global conve}
There exists $k_0\in \N$ such that $C_k=C_{k_0}$ and $t^{\min}_k=t^{\min}_{k_0}$ for all $k\ge k_0.$
\end{Claim}
\noindent Arguing by contradiction, suppose that such a number $k_0$ does not exist. By the construction of $\set{C_k}$ and $\set{t^{\min}_k}$ in Step~3 of Algorithm~\ref{GDF C11}, we deduce that $C_k\uparrow \infty$ and $t^{\min}_k\downarrow 0$ as $k\to\infty$. Since $\Omega$ is bounded, Lemma~\ref{lemm t uniform} allows us to find $\bar t\in(0,1)$ for which
\begin{align}\label{linesearch proof}
f(x-tg)\le f(x)-\beta t\norm{g}^2\;\text{ whenever }\; x\in \Omega,\;\norm{g-\nabla f(x)}\le\dfrac{1}{2}\norm{g},\;\text{ and }\;t\in(0,\bar t].
\end{align}
Using the aforementioned properties of $\{C_k\}$ and 
$\{t^{\min}_k\}$, we get $N\ge K$ such that $C_k>C$ and $t^{\min}_k<\gamma\bar t$ for all $k\ge N$. Fix such a number $k$ and then combine the condition $\norm{g^k}>\mu C_k\delta_{k+1}$ from \eqref{choice gk} with $C_k>C$, $\mu>2$, and \eqref{equa norm gk}. This gives us 
\begin{align*}
\norm{g^k}>\mu C_k\delta_{k+1}\ge \mu C\delta_{k+1}\ge 2 \norm{g^k-\nabla f(x^k)},
\end{align*}
which implies together with $x^k\in\Omega$ and \eqref{linesearch proof} the estimate
\begin{align*}
f(x^k-tg^k)\le f(x^k)-\beta t\norm{g^k}^2 \text{ for all }t\in(0,\bar t]
\end{align*}
and thus tells us that $t_k>\gamma\bar t>t_k^{\min}$. Employing Step~3 of Algorithm~\ref{GDF C11} yields $t_{k+1}^{\min}=t_k^{\min}$. Since the latter holds whenever $k\ge N$, we conclude that the equality $t_k^{\min}=t_N^{\min}$ is satisfied for all $k\ge N$. This contradicts the condition $t^{\min}_k\downarrow0$ as $k\to\infty$ and hence justifies the claimed assertion.
\begin{Claim}
All the assertions in {\rm(i)} and {\rm(ii)} are fulfilled.
\end{Claim}
\noindent From Step 2 and Step 3 of Algorithm~\ref{GDF C11}, we deduce that 
\begin{align}\label{first ineq}
f(x^k)-f(x^{k+1})\ge \beta \tau_k\norm{g^k}^2\;\text{ for all }\;k\in\N.
\end{align}
Defining $N:=\max\set{K,k_0}$ with $k_0$ taken from Claim~\ref{cl 1 global conve} gives us the equalities
\begin{align}\label{Ck=CN}
C_k=C_N\;\text{ and }\;t_k^{\min}=t_N^{\min}\;\text{ whenever }\;k\ge N.
\end{align}
Combining $C_k=C_N$ with \eqref{equa norm gk} and $\norm{g^k}>\mu C_k\delta_{k+1}$ from \eqref{choice gk} ensures that
\begin{align}\label{second ineq}
\norm{\nabla f(x^k)}&\le \norm{g^k}+C\delta_{k+1} \nonumber \\
&\le \norm{g^k}+\dfrac{C}{\mu C_N}\norm{g^k} =\alpha\norm{g^k} \;\text{ for all }\;k\ge N,
\end{align}
where $\alpha:=1+\frac{C}{\mu C_N}$. In addition, we have $t_{k+1}^{\min}=t_{k}^{\min}=t_N^{\min}$ in \eqref{Ck=CN}, which implies together with Step~3 of Algorithm~\ref{GDF C11} the relationships
\begin{align}\label{third ineq}
\tau_k=t_k\ge t_k^{\min}=t_N^{\min}\;\text{ as }\;k\ge N
\end{align}
confirming the boundedness of $\set{\tau_k}$ from below. If the KL property of $f$ holds at the accumulation point $\bar x$ of $\set{x^k}$, it follows from Remark~\ref{rmk two KL}(i), \eqref{first ineq}, \eqref{second ineq}, and \eqref{third ineq} that assumptions (H1) and (H2) in Proposition~\ref{general convergence under KL} hold. Thus $x^k\rightarrow\bar x$ as $k\to\infty$, which verifies (i).

Assume finally that the KL property at $\bar x$ is satisfied with $\psi(t)=Mt^{q},\;M>0$, and $q\in(0,1)$. The iterative procedure $x^{k+1}=x^k-\tau_kg^k$ in Step~4 of Algorithm~\ref{GDF C11} together with \eqref{third ineq} and $g^k>0$ from Step~1 therein tells us that $x^{k+1}\ne x^k$ for $k\ge N.$ Combining this with \eqref{first ineq}, \eqref{second ineq}, and \eqref{third ineq} verifies all the assumptions of Proposition~\ref{general rate} and therefore completes the proof of the theorem.
\end{proof}\vspace*{-0.2in}

\section{Local Convergence of General Derivative-Free Methods to Nonisolated Local Minimizers}\label{sec:6}
In this section, we develop appropriate versions of both Algorithm~\ref{GDF} and Algorithm~\ref{GDF C11} and obtain efficient conditions that guarantee the {\em local convergence} with constructive convergence rates of these algorithms to {\em nonisolated local minimizers}.

The following general algorithm of this type covers local versions of both Algorithms~\ref{GDF} and \ref{GDF C11}. 

\begin{longfbox}
\begin{Algorithm}[General derivative-free (GDF) method]\hlabel{GDF general}
\setcounter{Step}{-1}
\begin{Step}\rm Choose a local approximation $\mathcal{G}$ of $\nabla f$ under condition \eqref{local approx}. Select an initial point $x^1\in\R^n,$ an initial sampling radius $\delta_1>0,$ a sequence $\set{C_k}\subset \R_+,$ a reduction factor $\theta\in(0,1)$ and a scaling factor $\mu>2$. Choose a sequence of manually controlled sampling radii $\set{\nu_k}\subset[0,\infty)$. Set $k:=1$.
\end{Step}
\begin{Step}[approximate gradient]\rm Find $g^k$ and the smallest nonnegative integer $i_k$ such that
\begin{align*}
g^k=\mathcal{G}\big(x^k,\min\set{\theta^{i_k}\delta_k,\nu_k}\big)\;
\text{ and }\;\norm{g^k}>\mu C_k\theta^{i_k}\delta_k.
\end{align*}
Then set $\delta_{k+1}:=\theta^{i_k}\delta_k.$
\end{Step}
\begin{Step}[update]\rm
Choose a stepsize $\tau_k\ge 0$ and set $x^{k+1}:=x^k-\tau_kg^k.$ Go back to Step 1.
\end{Step}
\end{Algorithm}
\end{longfbox}\\[2ex]
Similarly to Remark~\ref{Remark C11}, the existence of $g^k$ in Step~1 of Algorithm~\ref{GDF general} is guaranteed if $\nabla f(x^k)\ne 0$. The result addressing the local convergence to local minimizers of Algorithm~\ref{GDF general} is presented below.
\begin{Theorem}\label{GDF local theorem}
Let $f:\R^n\rightarrow\R$ be a $\mathcal{C}^1$-smooth function, let $\bar x\in\R^n$, and let $\Delta>0.$ Assume that $\bar x$ is a local minimizer of $f$ satisfying the KL property at $\bar x$, and that $\nabla f$ is locally Lipschitz continuous around $\bar x$. Then we have the assertions:

{\bf(i)} There exist positive numbers $\xi$, $T$, and $C$ such that for any initial point $x^1 \in \mathbb{B}(\bar{x}, \xi)$, an initial radius $\delta_1 \in (0, \Delta]$, and sequences $\set{\tau_k} \subset [0, T ]$, $\set{C_k} \subset [C, \infty)$, and $\set{\nu_k}\subset [0,\infty)$. it holds we that the iterative sequence $\set{x^k}$ of Algorithm~{\rm\ref{GDF general}} converges provided that $\nabla f(x^k)\ne 0$ for all $k\in\N$.

{\bf(ii)} If in addition the series $\sum_{k=1}^\infty \tau_k$ diverges, then the sequence $\set{x^k}$ converges to a local minimizer $\tilde{x}$ of $f$ with $f(\tilde{x})=f(\bar x)$. Moreover, the following convergence rates are guaranteed if the sequence $\set{\tau_k}$ is bounded away from $0$, and if the KL property is satisfied with $\psi(t)=Mt^q,\;M>0$, and $q\in(0,1)$:\\
$\bullet$ Whenever $q\in(0,1/2]$, the sequences $\set{x^k},\set{\nabla f(x^k)}$, and $f(x^k)$ converge linearly s $k\to\infty$ to the points $\tilde x$, $0\in\R^n$, and $f(\bar x)$, respectively.\\
$\bullet$ For $q\in(1/2,1)$, we get the estimates $\norm{x^k-\tilde x}=\mathcal{O}( k^{-\frac{1-q}{2q-1}})$, $\norm{\nabla f(x^k)}=\mathcal{O}(k^{-\frac{1-q}{2q-1}})$, and $f(x^k)-f(\bar x)=\mathcal{O}(k^{-\frac{2-2q}{2q-1}})$  as $k\to\infty$.        
\end{Theorem}
\begin{proof} Since  $f$ satisfies the KL property at $\bar x$, there exist a bounded neighborhood $U$ of $\bar x$, a number $\eta>0$, and a nonincreasing function $\psi:(0,\eta)\rightarrow(0,\infty)$ such that $1/\psi$ is integrable over $(0,\eta)$ and
\begin{align}\label{KL in proof local}
\norm{\nabla f(x)}\ge \psi(f(x)-f(\bar x))\;\text{ for all }\;x\in  U \;\text{ with }\;f(\bar x)< f(x)<f(\bar x)+\eta.
\end{align}
Remark~\ref{algebraic} tells us that \eqref{KL in proof local} also holds if $\bar x$ is replaced by any $\tilde x\in U$ with $f(\tilde x)=f(\bar x)$. Since $\bar x$ is a local minimizer of $f$ and since $f$ is continuous, we can assume by shrinking $ U$ if necessary that $f(\bar x)\le f(x)<f(\bar x)+\eta$ for all $x\in U$. Combining this with \eqref{KL in proof local} ensures that 
\begin{align}\label{KL ineq}
\norm{\nabla f(x)}\ge \psi(f(x)-f(\bar x))>0\;\text{ for all }\;x\in  U\;\text{ with }\;f( x)\ne f(\bar x),
\end{align}
which implies in turn that $f(\bar x)$ is the only critical value of $f$ within $U$. The local Lipschitz continuity of $\nabla f$ around $\bar x$ gives us positive numbers $\rho$ and $L$ such that $\mathbb{B}(\bar x,2\rho)\subset U$ and $\nabla f$ is Lipschitz continuous with the constant $L$ on $\mathbb{B}(\bar x,2\rho)$. Choose further $T :=\min\set{\frac{1}{3L},\frac{\mu-2}{\mu L}}$ and define $\varphi:[0,\eta)\rightarrow[0,\infty)$ by $\varphi(x):=\int_0^x\frac{1}{\psi(t)}dt$ for $x\in(0,\eta)$ with $\varphi(0):=0$. By the right continuity of $\varphi$ at $0$ and the continuity of $f$ at $\bar x$, we find such $\xi\in(0,\rho)$ that 
\begin{align}\label{selection xi}
\norm{x-\bar x}+4\varphi(f(x)-f(\bar x))<\rho\;\text{ for all }\;x\in \mathbb{B}(\bar x,\xi).
\end{align}
Since $\mathcal{G}$ is a local approximation of $\nabla f$ under condition \eqref{local approx}, we find $C>0$ for which 
\begin{align}\label{in proof local approx}
\norm{\mathcal{G}(x,\delta)-\nabla f(x)}\le C\delta\;\text{ whenever }\;(x,\delta)\in U\times(0,\Delta].
\end{align} 

Assuming that $x^1\in \mathbb{B}(\bar x,\xi)$, $\delta_1\in  (0,\Delta]$, $\set{\tau_k}\subset [0,T]$, and $\set{C_k}\subset[C,\infty)$, we now aim at verifying (i). To proceed, let us first prove the following claim. 
\setcounter{Claim}{0}
\begin{Claim}\label{cl 1 local}  Algorithm~{\rm\ref{GDF general}} generates the well-defined iterative sequence $\set{x^k}$, which stays inside $\mathbb{B}(\bar x,\rho)$.
\end{Claim}
\noindent 
Indeed, by $\nabla f(x^k)\ne 0$ for all $k\in\N$, the existence of $g^k$ in Step~2 of Algorithm~\ref{GDF general} is guaranteed in each iteration, which ensures that the iterative sequence $\set{x^k}$ is well-defined. To verify that $\set{x^k}\subset\mathbb{B}(\bar x,\rho)$, we proceed by induction. Fix $K\in\N$ and assume that $x^k\in\mathbb{B}(\bar x,\rho)$ for all $k=1,\ldots,K$. To show that $x^{K+1}\in \mathbb{B}(\bar x,\rho)$, observe that $\delta_k\le \delta_1\le \Delta$ for all $k\in\N$ by the selection of $\delta_1$ and the construction of $\set{\delta_k}$ in Algorithm~\ref{GDF general}. Since $\set{C_k}\subset[C,\infty)$, we deduce from \eqref{in proof local approx} and $g^k=\mathcal{G}(x^k,\min\set{\delta_{k+1},\nu_k})$ in Step~1 of Algorithm~\ref{GDF general} that
\begin{align}\label{ineq IGD}
\norm{g^k-\nabla f(x^k)}\le C\min\set{\delta_{k+1},\nu_k}\le C_k\delta_{k+1} \le\mu^{-1}\norm{g^k}\;\text{ for all }\;k=1,\ldots,K.
\end{align}
It follows from \eqref{ineq IGD} with $k:=K$, the triangle inequality, and the choice of $\mu>2$ that
\begin{align}\label{gk 2 nabla}
\begin{array}{ll}
\norm{\nabla f(x^K)}\ge\norm{g^K}-\norm{g^K-\nabla f(x^K)}>(1-\mu^{-1})\norm{g^K}>\frac{1}{2}\norm{g^K}.
\end{array}
\end{align} 
Since $x^K\in \mathbb{B}(\bar x,\rho)$, we deduce from the Lipschitz continuity of $f$ on $\mathbb{B}(\bar x,2\rho)$ that 
$$
\norm{\nabla f(x^K)}=\norm{\nabla f(x^K)-\nabla f(\bar x)}\le L\norm{x^K-\bar x}\le L\rho.
$$  
Combining this with the update $x^{k+1}=x^k-\tau_kg^k$, $\tau_k\le T \le(3L)^{-1}$, and \eqref{gk 2 nabla} gives us
\begin{align*}
\norm{x^{K+1}-x^K}&=\tau_k\norm{g^K}\le 2T \norm{\nabla f(x^K)}\le 2T L\rho< \rho,
\end{align*}
which means that $x^{K+1}\in\mathbb{B}(\bar x,2\rho)$.
Since $\nabla f(x^{K+1})\ne 0$ and $f(\bar x)$ is the minimum value of $f$ within $\mathbb{B}(\bar x,2\rho)$, we get that $f(x^{K+1})>f(\bar x)$. It follows from the triangle inequality and estimate \eqref{ineq IGD} that
\begin{align*}
\norm{g^k}\ge \norm{\nabla f(x^k)}-\norm{g^k-\nabla f(x^k)}\ge \norm{\nabla f(x^k)}-\mu^{-1} \norm{g^k},
\end{align*}
which implies in turn that
\begin{align}\label{nabla f(xk)<=2gk}
\norm{\nabla f(x^k)}\le(1+\mu^{-1})\norm{g^k}\le 2\norm{g^k}\;\text{ for all }\;k=1,\ldots,K. 
\end{align}
The Lipschitz continuity of $\nabla f$ on $\mathbb{B}(\bar x,2\rho)\supset \set{x^k\;|\;k=1,\ldots,K+1}$ with the constant $L>0$ yields
\begin{subequations}
\begin{align}
f(x^{k+1})-f(x^k)&\le \dotproduct{\nabla f(x^k),x^{k+1}-x^k}+\dfrac{L}{2}\norm{x^{k+1}-x^k}^2\label{eq: descent property}\\
&=-\tau_k\dotproduct{\nabla f(x^k),g^k}+\dfrac{L\tau_k^2}{2}\norm{g^k}^2\label{eq: iterative update}\\
&=-\tau_k\dotproduct{\nabla f(x^k)-g^k,g^k}-\tau_k\norm{g^k}^2+\dfrac{L\tau_k^2}{2}\norm{g^k}^2\nonumber\\
&\le \tau_k\norm{\nabla f(x^k)-g^k}\norm{g^k}-\tau_k\norm{g^k}^2+\dfrac{L\tau_k^2}{2}\norm{g^k}^2\label{eq: cauchy schwarz}\\
&\le \tau_k\frac{1}{\mu}\norm{g^k}^2-\tau_k\norm{g^k}^2+\dfrac{L\tau_k^2}{2}\norm{g^k}^2\label{eq: tk nu gk^2}\\
&\le \norm{\tau_kg^k}\norm{g^k}\Big(\frac{1}{\mu}-1+\dfrac{L\tau_k}{2}\Big)\nonumber\\
&\le -\dfrac{1}{2}\norm{x^{k+1}-x^k}\norm{g^k}\label{eq: choice of tk}\\
&\le-\dfrac{1}{4}\norm{x^{k+1}-x^k}\norm{\nabla f(x^k)}\;\text{ for all }\;k=1,\ldots,K,\label{norm norm}
\end{align}
\end{subequations}
where \eqref{eq: descent property} follows from Lemma~\ref{lemma descent}, \eqref{eq: iterative update} follows from the iterative update $x^{k+1}=x^k-\tau_kg^k$, \eqref{eq: cauchy schwarz} follows from the Cauchy-Schwarz inequality, \eqref{eq: tk nu gk^2} is deduced by $\norm{g^k-\nabla f(x^k)}\le \frac{1}{\mu}\norm{g^k}$ from \eqref{ineq IGD}, \eqref{eq: choice of tk} follows from  $\tau_k\le T \le\dfrac{\mu-2}{\mu L}$, and \eqref{norm norm} is deduced by $\norm{\nabla f(x^k)}\le 2\norm{g^k}$ from \eqref{nabla f(xk)<=2gk}. Therefore,
\begin{align}\label{1 4 descent}
\dfrac{1}{4}\norm{x^{k+1}-x^k}\norm{\nabla f(x^k)} \le f(x^k)-f(x^{k+1})\text{ whenever }k=1,\ldots,K.
\end{align}
As a consequence of the above, we have the inequalities
\begin{align*}
f(\bar x)+\eta> f(x^k)\ge f(x^{K+1})>f(\bar x)\;\text{ for all }\;k=1,\ldots,K+1,
\end{align*}
which ensure together with $x^k\in \mathbb{B}(\bar x,2\rho)$ and \eqref{KL ineq} that $\norm{\nabla f(x^k)}\ge \psi(f(x^k)-f(\bar x))$ for all $k=1,\ldots,K$. Combining the latter with \eqref{1 4 descent} leads us to the conditions
\begin{align*}
\dfrac{1}{4}\norm{x^{k+1}-x^k}&\le \dfrac{f(x^k)-f(x^{k+1})}{\psi(f(x^k)-f(\bar x))}=\int_{f(x^{k+1})}^{f(x^{k})}\dfrac{1}{\psi(f(x^k)-f(\bar x))}dt\\
&\le \int_{f(x^{k+1})}^{f(x^{k})}\dfrac{1}{\psi(t-f(\bar x))}dt\\
&=\varphi(f(x^k)-f(\bar x))-\varphi(f(x^{k+1})-f(\bar x))\;\text{ for all }k=1,\ldots,K,
\end{align*}
where the second inequality follows from the nondecreasing property of $\psi.$ Therefore, the triangle inequality gives us the relationships
\begin{align*}
\norm{x^{K+1}-\bar x}&\le \norm{x_1-\bar x}+\sum_{k=1}^K\norm{x^{k+1}-x^{k}}\\
&\le \norm{x_1-\bar x}+4\sum_{k=1}^K[\varphi(f(x^{k})-f(\bar x))-\varphi(f(x^{k+1})-f(\bar x))]\\
&=\norm{x_1-\bar x}+4\sbrac{\varphi(f(x^1)-f(\bar x))-\varphi(f(x^{K+1})-f(\bar x))}\\
&\le \norm{x_1-\bar x}+4\varphi(f(x^1)-f(\bar x))<\rho,
\end{align*}
where the latter inequality is a consequence of the selection $x^1\in\mathbb{B}(\bar x,\xi)$ and $\xi$ in \eqref{selection xi}. This means that $x^{K+1}\in\mathbb{B}(\bar x,\rho)$. By induction we arrive at $x^k\in \mathbb{B}(\bar x,\rho)$ for all $k\in\N$, which verifies this claim. 

\begin{Claim}
The sequence of iterates $\set{x^k}$ converges to some $\tilde x\in\overline{\mathbb{B}}(\bar x,\rho)$. If in addition we have $\sum_{k=1}^\infty \tau_k=\infty$, then $\tilde x$ is a local minimizer of $f$. 
\end{Claim}
\noindent
Picking any $K\in\N$ and arguing similarly to the proof of Claim~\ref{cl 1 local} with taking into account that $x^k\in\mathbb{B}(\bar x,\rho)$ for all $k\in\N$, we get the estimates
\begin{align*}
\sum_{k=1}^K\norm{x^{k+1}-x^k}\le4\sum_{k=1}^K[\varphi(f(x^{k+1})-f(\bar x))-\varphi(f(x^{k})-f(\bar x))]\le 4\varphi(f(x^1)-f(\bar x)).
\end{align*}
Passing there to the limit as $K\rightarrow\infty$ yields $\sum_{k=1}^\infty \norm{x^{k+1}-x^k}<\infty$, which tells us that $\set{x^k}$ converges to some point $\tilde x\in\overline{\mathbb{B}}(\bar x,\rho)$. Let now $\sum_{k=1}^\infty \tau_k=\infty$ be satisfied. Since $\set{x^k} \subset \mathbb{B}(\bar x,\rho)$, we proceed similarly to the proof of \eqref{eq: choice of tk} in Claim~\ref{cl 1 local} to show that
\begin{align}\label{descent condition KL proof}
f(x^{k+1})-f(x^k)&\le -\dfrac{1}{2}\norm{x^{k+1}-x^k}\norm{g^k}\;\text{ for all }\;k\in\N.
\end{align}
Combining this with the fact that $f(x^k)\ge f(\bar x)$ for all $k\in\N$, we deduce that 
\begin{align*}
\sum_{k=1}^\infty \tau_k\norm{g^k}^2=\sum_{k=1}^\infty \norm{x^{k+1}-x^k}\norm{g^k}&\le 2\sum_{k=1}^\infty (f(x^k)-f(x^{k+1}))\\
&\le 2(f(x^1)-f(\bar x))<\infty.
\end{align*}
Supposing that there is $r>0$ with $\norm{g^k}\ge r$ for all $k$ sufficiently large, the above inequality gives us $\sum_{k=1}^\infty \tau_k<\infty$, which is a contradiction. Therefore, $0\in\R^n$ is an accumulation point of $\set{g^k}$, i.e., there exists an infinite set $J\subset \N$ such that $g^k\overset{J}{\rightarrow}0$. As in the proof of \eqref{ineq IGD} in Claim~\ref{cl 1 local} with taking now $x^k\in \mathbb{B}(\bar x,\rho)$ into account, we get the estimate
\begin{align*}
\norm{g^k-\nabla f(x^k)} \le\mu^{-1}\norm{g^k}\;\text{ for all }\;k\in\N.
\end{align*}
Combining the latter with $g^k\overset{J}{\rightarrow}0$ tells us that $\nabla f(x^k)\overset{J}{\rightarrow}0$. Remembering that $\set{x^k}$ converges to $\tilde x$, we have that $\nabla f(x^k)\rightarrow \nabla f(\tilde x)$ as $k\to\infty$. Therefore, $\nabla f(\tilde x)=0$, i.e., $\tilde{x}$ is a stationary point of $f$ on $\mathbb{B}(\bar x,2\rho)$. Since $f(\bar x)$ is the only critical value of $f$ within $\mathbb{B}(\bar x,2\rho)$, we obtain that $f(\tilde x)=f(\bar x)$, which tells us that $\tilde x$ is a local minimizer of $f$.

Invoke now the assumptions that the sequence $\set{\tau_k}$ is bounded away from $0$ and that the KL property \eqref{KL in proof local} holds with $\psi(t)=Mt^q,\;M>0,$ and$q\in(0,1)$. It follows from Steps~1 and 2 of Algorithm~\ref{GDF general} that $x^{k+1}\ne x^k$ for large $k$. By \eqref{descent condition KL proof} and the iterative procedure $x^{k+1}=x^k-\tau_k g^k$, we get
\begin{align*}
f(x^k)-f(x^{k+1})&\ge \dfrac{1}{2}\tau_k\norm{g^k}^2\;\text{ whenever }\;k\in\N.
\end{align*}    
Arguing similarly to \eqref{nabla f(xk)<=2gk} with taking into account that $x^k\in\mathbb{B}(\bar x,\rho)$ for all $k\in\N$, we deduce that
\begin{align*}
\norm{\nabla f(x^k)}\le 2 \norm{g^k}\;\text{ as }\;k\in\N.
\end{align*}
Therefore, the convergence rates for $\set{x^k}$ in (ii) are immediately obtained by applying Proposition~\ref{general rate} with observing that the KL property \eqref{KL in proof local} also holds with $\psi(t)=Mt^q$, $M>0$, and $q\in(0,1)$ when $\bar x$ is replaced by $\tilde x$. By arguing similarly to the proof of the convergence rates in Theorem~\ref{convergence GDF} and remembering that $\nabla f$ is Lipschitz continuous on $\mathbb{B}(\bar x,2\rho)\supset\set{x^k}$, we verify the desired convergence rates for $\{\nabla f(x^k)\}$ and $\{f(x^k)\}$ in (ii) and thus complete the proof of the theorem. 
\end{proof}
\begin{Remark}\rm
Note that the condition $\nabla f(x^k) \neq 0$ for all $k \in \mathbb{N}$ in Theorem \ref{GDF local theorem} is a standard assumption in the convergence analysis of derivative-free optimization methods since there is no tool available to determine whether $\nabla f(x^k)$ equals zero or not. Similar assumptions can also be found at \cite[Section~4]{conn09} and \cite[Corollary~3.3]{hare13}. Recently, Josz et al. \cite{josz23.2} presented a local convergence analysis for {\em exact momentum methods} with \textit{constant stepsizes}, specifically focusing on {\em semi-algebraic} functions with gradients being {\em locally Lipschitzian everywhere}. It is necessary to emphasize that this analysis does not encompass the obtained convergence properties in Theorem~\ref{GDF local theorem}. Specifically, our work focuses on {\em derivative-free methods} with {\em variable stepsizes} when applied to $\mathcal{C}^{1}$-smooth functions satisfying the KL property. These functions have gradients that are {\em locally Lipschitzian} but only in the vicinity of {\em local minimizers}. Observe also that the local convergence result above does not follow from the one in  \cite[Theorem~2.10]{attouch13}. The latter relies on global conditions (H1) and (H2), along with the local growth condition (H4), which are not presumed in our derivative-free context and may not be satisfied within the scope of our given assumptions.
\end{Remark}
The following two consequences of Theorem~\ref{GDF local theorem} address the local convergence to local minimizers of Algorithm~\ref{GDF} and Algorithm~\ref{GDF C11}. 
\begin{Corollary}\label{corollary 1}
Let $f:\R^n\rightarrow\R$ be a $\mathcal{C}^1$-smooth function with a globally Lipschitz continuous gradient, let $\bar x\in\R^n$, and let $\Delta>0$. Assume that $\bar x$ is a local minimizer of $f$ satisfying the KL property at $\bar x$, and that $\nabla f(x^k)\ne 0$ for all $k\in\N$. Then there exist $\xi,C>0$ such that for any initial point $x^1\in \mathbb{B}(\bar x,\xi)$, any initial sampling radius $\delta_1\in  (0,\Delta]$, any $C_1 \ge C$ and other parameters listed in Algorithm~{\rm\ref{GDF}}, we have that $\set{x^k}$ converges to a local minimizer $\tilde{x}$ of $f$ with $f(\tilde{x})=f(\bar x)$. The convergence rates as in Theorem~{\rm\ref{GDF local theorem}(ii)} are guaranteed if $f$ satisfies the KL property at $\bar x$ with $\psi(t)=Mt^{q}$, $M>0$, and $q\in (0,1)$.
\end{Corollary}
\begin{proof}
By Theorem~\ref{GDF local theorem}, there are $\xi,T,C>0$ such that for any initial point $x^1\in \mathbb{B}(\bar x,\xi)$, initial radius $\delta_1\in (0,\Delta]$, and sequences ${\tau_k}\subset [0,T]$ and ${C_k}\subset [C,\infty)$, it holds that any sequence of iterates generated by Algorithm~\ref{GDF general} exhibits the convergence properties presented in Theorem~\ref{GDF local theorem}. By choosing a larger $C$ if necessary, we can assume that $\kappa/C\le T $, where $\kappa>0$ is the parameter taken from Algorithm~\ref{GDF}.

It suffices to show that Algorithm~\ref{GDF} with initial point $x^1\in\mathbb{B}(\bar x,\xi)$, $C_1\ge C,\;\delta_1\in(0,\Delta]$, and the other parameters therein is a special case of Algorithm~\ref{GDF general} with the parameters listed above, and thus  Algorithm~\ref{GDF} enjoys the desired convergence properties. It is clear from the construction of Algorithm~\ref{GDF} that $C_{k+1}\ge C_k$ for all $k\in\N$, which implies that $C_k\ge C_1\ge C$ whenever $k\in\N$. The selection of $\set{g^k}$ in Step~1 of Algorithm~\ref{GDF general} reduces to that of Algorithm~\ref{GDF} by choosing $\nu_k\ge\delta_1$ for all $k\in\N$. The iterative procedure of Algorithm~\ref{GDF} can be rewritten as
\begin{align*}
x^{k+1}=x^k-\tau_kg^k\;\text{ with either }\;\tau_k=0,\;\text{ or }\;\tau_k=\kappa/C_k\le\kappa/C\le T\;\text{ for all }\;k\in\N
\end{align*}
telling us that $\set{\tau_k}\subset[0,T]$. Proposition~\ref{claim 1 sec 5} guarantees that $C_{k}$ are constant for large $k\in\N$, which ensures by Step~2 of Algorithm~\ref{GDF} that the stepsizes $\tau_k$ are equal to a positive constant for large $k\in\N$. This guarantees that the sequence $\set{\tau_k}$ is bounded away from $0$, and furthermore $\sum_{k=1}^\infty \tau_k=\infty$. All the assumptions in Theorem~\ref{GDF local theorem} are satisfied. Thus $\set{x^k}$ converges to some local minimizer $\tilde x$ of $f$ with $f(\tilde x)=f(\bar x)$,
and in addition the convergence rates as in Theorem~\ref{GDF local theorem}(ii) are guaranteed when $f$ satisfies the KL property at $\bar x$ with $\psi(t)=Mt^{q}$, $M>0$, and $q\in (0,1)$.
\end{proof}
\begin{Corollary}
Let $f:\R^n\rightarrow\R$ be a $\mathcal{C}^1$-smooth function with a locally Lipschitzian gradient, let $\bar x\in\R^n$, and let $\Delta>0$. Assume that $\bar x$ is a local minimizer of $f$, which satisfies the KL property at $\bar x$, and that $\nabla f(x_k)\ne 0$ for all $k\in\N$. Then there are constants $\xi,T,C>0$ such that for any initial point $x^1\in\mathbb{B}(\bar x,\xi)$, any initial sampling radius $\delta_1\in (0,\Delta]$, $\bar \tau \in (0,T], C_1 \ge C$, and the other parameters of Algorithm~{\rm\ref{GDF C11}}, the sequence of iterates $\set{x_k}$ in this algorithm converges to some local minimizer $\tilde{x}$ with $f(\tilde x)=f(\ox)$. The convergence rates as in Theorem~{\rm\ref{GDF local theorem}(ii)} are guaranteed if the sequence $\set{\tau_k}$ is bounded away from $0$, and if $f$ satisfies the KL property at $\bar x$ with $\psi(t)=Mt^{q}$, $M>0$, and $q\in (0,1)$.
\end{Corollary}
\begin{proof}
By Theorem~\ref{GDF local theorem}, there exist positive numbers  $\xi,T,C$ such that for any initial point $x^1\in \mathbb{B}(\bar x,\xi)$, any initial radius $\delta_1\in (0,\Delta]$, and any sequences $\set{\tau_k}\subset [0,T ]$ and $\set{C_k}\subset[C,\infty)$, the sequence of iterates of Algorithm~\ref{GDF general} exhibits the properties listed in Theorem~\ref{GDF local theorem}(i,ii).

It is sufficient to verify that Algorithm~\ref{GDF C11} with initial point $x^1\in\mathbb{B}(\bar x,\xi)$, $C_1\ge C,\bar \tau\in(0,T],\delta_1\in(0,\Delta]$, and with other parameters taken from Algorithm~\ref{GDF C11} is a special case of the general Algorithm~\ref{GDF general}, and hence it enjoys the claimed convergence properties. We see from the structure of Algorithm~\ref{GDF C11} that $C_{k+1}\ge C_k$ for all $k\in\N$, which tells us that $C_k\ge C$ for all $k\in\N$.

It follows from  Step~2 of Algorithm~\ref{GDF C11} that $\tau_k\le \bar \tau$ for all $k\in\N$. Combining this with $\bar \tau \leq T$, ensures that ${\tau_k} \subset [0, T]$. Moreover, the selection of $\set{g^k}$ in Step~1 of Algorithm~\ref{GDF general} also reduces to that of Algorithm~\ref{GDF C11} since $\set{\nu_k}\downarrow 0$ as $k\to\infty$. Theorem~\ref{GDF local theorem}(i) tells us that the sequence of iterates $\set{x^k}$ generated by Algorithm~\ref{GDF C11} converges, and thus it is bounded. By using \eqref{third ineq} as in the proof of Theorem~\ref{KL GDF C11}, we get that the sequence of stepsizes $\set{\tau_k}$ is bounded away from $0$. This guarantees the condition $\sum_{k=1}^\infty \tau_k=\infty$ in Theorem~\ref{GDF local theorem}(ii), which verifies therefore that $\set{x^k}$ converges to some local minimizer $\tilde x$ of $f$ with the convergence rates as in Theorem~\ref{GDF local theorem}(ii). Thus the proof is complete.  
\end{proof}

If the condition $\sum_{k=1}^\infty \tau_k=\infty$ is removed, the trivial example with $\tau_k=0$ for all $k\in\N$ can be used to show that the iterative sequence $\set{x^k}$ generated by Algorithm~\ref{GDF general} may remain at the given initial point, which is nonstationary, and thus it does not converge to any local minimizer of $f.$ Even when $\tau_k>0$ for all $k\in\N$ but $\sum_{k=1}^\infty \tau_k<\infty$, the following example shows that $\set{x^k}$ may still converge to a nonstationary point, which confirms an essential role of the assumption $\sum_{k=1}^\infty \tau_k=\infty$ in deriving the local optimality of $\tilde x$ in Theorem~\ref{GDF local theorem}.
\begin{Example}\rm
Considering the function $f(x):=x^2$, it is clear that its derivative $f'$ is globally Lipschitz continuous, and that $f$ satisfies the KL property at the local minimizer $0$ with $\psi(t):=2t^{\frac{1}{2}}$. Therefore, all the assumptions in Theorem~\ref{GDF local theorem} are satisfied except for $\sum_{k=1}^\infty \tau_k=\infty$. Now we consider Algorithm~\ref{GDF general} with $\mathcal{G}(x,\delta)$ being chosen via the central finite difference. Then for any $x\in\R$, we have
\begin{align}\label{constant derivative}
\mathcal{G}(x,\delta)=\dfrac{f(x+\delta)-f(x-\delta)}{2\delta}=\dfrac{(x+\delta)^2-(x-\delta)^2}{2\delta}=\dfrac{4x\delta}{2\delta}=2x=f'(x)\text{ whenever }\delta>0.
\end{align}
Let us show that for any positive numbers $\xi,T$ and $C$, there exists an iterative sequence $\set{x^k}$ generated by Algorithm~\ref{GDF general} with the initial point $x^1\in \mathbb{B}(0,\xi)$, with the sequence of stepsizes $\set{\tau_k}\subset(0,T ]$, and with $C_k=C$ for all $k\in \N$ such that $\set{x^k}$ converges to a \textit{nonstationary point} of $f$.
To proceed, take any $\xi,T,C>0$ and choose $x^1:=\xi/2\in \mathbb{B}(0,\xi)$. The sequence $\set{\tau_k}$ is constructed inductively as follows. For any $k\in\N$ with $x^k>x^1/2>0$, we deduce from Step~1 of Algorithm~\ref{GDF general} and \eqref{constant derivative} that
\begin{align*}
g^k=\mathcal{G}(x^k,\delta_{k+1})=f'(x^k)=2x^k\;\text{ for all }\;k\in\N.
\end{align*}
Choosing $\tau_k:=\min\big\{T ,\frac{1}{4}-\frac{x^1}{8x^k}\big\}>0$ tells us that
\begin{align*}
x^{k+1}&=x^k-\tau_kg^k=(1-2\tau_k)x^k\ge\Big(1-\frac{1}{2}+\frac{x^1}{4x^k}\Big)x^k\\
&=\dfrac{x^k}{2}+\dfrac{x^1}{4}>\dfrac{x^1}{2}.
\end{align*}
Arguing by induction with the usage of $x^1>\frac{x^1}{2}$ ensures that the sequence of positive stepsizes $\set{\tau_k}$ constructed as above provides $x^k>\frac{x^1}{2}$ for all $k\in\N$. This implies that $\set{x^k}$ does not converge to $0,$ which is the only stationary point of $f$. We can actually find the exact limit for $\displaystyle{\lim_{k\to\infty}}x^k$ to see that it is clearly not $0$. Since $x^k>0$ for all $k\in\N$, the iterative procedure $x^{k+1}=x^k-2\tau_kx^k$ shows that $\set{x^k}$ is decreasing, and thus it has a limit $\bar x$ by taking into account its boundedness from below. Therefore,
\begin{align*}
\tau_k=\dfrac{x^k-x^{k+1}}{2x^k}\le \dfrac{x^k-x^{k+1}}{4x^1}\rightarrow0\;\text{ as }\;k\rightarrow\infty
\end{align*}
which implies by the selection of $\tau_k$ that $\displaystyle{\lim_{k\to\infty}}\frac{x^1}{8x^k}=\frac{1}{4}$, i.e., $\displaystyle{\lim_{k\to\infty}}x^k=\frac{x^1}{2}.$ It can also be observed that the assumption $\sum_{k=1}^\infty \tau_k=\infty$ is not satisfied in this example since $x^1 \tau_k\le 2\tau_k x^k=x^{k}-x^{k+1}$ yields
\begin{align*}
\sum_{k=1}^\infty\tau_k\le \dfrac{1}{x^1}\sum_{k=1}^\infty(x^k-x^{k+1})=\dfrac{1}{x^1}(x^1-\lim_{k\rightarrow\infty}x^k)<\infty.
\end{align*}
\end{Example}\vspace*{-0.1in}
 
\section{Numerical Experiments}\label{sec:7}
Here we present numerical experiments demonstrating the efficiency of our methods. This section is split into two subsections addressing different classes of problems with the usage of different algorithms.

\subsection{Experiments for \texorpdfstring{$\mathcal{C}^{1,1}_L$}{Lg} functions}
The first subsection compares the efficiency of our DFC methods with some other well-known derivative-free methods for minimizing $\mathcal{C}^{1,1}_L$ functions. The detailed information on the methods considered in these numerical experiments are as follows:
\begin{enumerate}[(i)]
\item DFC-fordif, i.e., Algorithm \ref{GDF} with forward finite difference.
\item  DFC-cendif, i.e., Algorithm \ref{GDF} with central finite difference.
\item FMINSEARCH, i.e., a Matlab implementation of the Nelder–Mead simplex-based method by Lagarias et al. \cite{lagarias98}.
\item  IMFIL-fordif, i.e., the implicit filtering algorithm with forward finite difference \cite{gilmore95}.
\item  IMFIL-cendif, i.e., the implicit filtering algorithm with central finite difference \cite{gilmore95}.
\item RG, i.e., a random gradient-free algorithm for smooth optimization proposed by Nesterov and Spokoiny \cite{nesterov17}.
\end{enumerate}
Note that the exact calculation of Lipschitz constants is not necessary for the implementation of FMINSEARCH, IMFIL, and our DFC methods while it is compulsory for RG. The solvers are tested on minimizing a noisy function $\phi(x):=f(x)+\xi(x)$, where $\xi(x)$ is a uniformly distributed random variable with a given noise level $\varepsilon\ge 0$, i.e., $\xi(x)\sim U(-\varepsilon,\varepsilon)$, and where the selections for $f(x)$ are as the following. 
\begin{enumerate}
\item {\em Least-square regression}
\begin{align}\label{least square}
f(x):=\norm{Ax-b}^2,
\end{align}
where $A$ is an $m\times n$ matrix and $b$ is a vector in $\R^m.$ This problem is used for benchmarking derivative-free optimization methods in \cite{rios13}.  It can be seen that $\nabla^2 f(x)=2A^*A\text{ for all }x\in\R^n,$ where $A^*$ is the transpose of matrix $A$. Therefore, $\nabla f$ is Lipschitz continuous with the constant $L:=2\norm{A^*A},$ where $\norm{M}:=\sigma_{\max}(M)$, which is the largest singular value of the matrix $M\in\R^{n\times n}$.
\item {\em Smooth variant of nonconvex image restoration}
\begin{align}\label{nonconvex}
f(x):=\sum_{i=1}^n \log(1+(Ax-b)^2_i),
\end{align}
where $A$ is a $n\times n$ matrix and $b$ is a vector in $\R^n.$ This problem is considered in \cite[Section~5.5]{themelis17} and \cite[Section~4]{kmpt23.2} with a nonsmooth term added to the objective function. It is clear that $\nabla f$ is Lipschitz continuous with the constant $L:=2\norm{A^*A}_\infty$, where
\begin{align*}
\norm{M}_\infty=\max_{1\le i\le n}\sum_{j=1}^n|m_{ij}|,\quad M=[m_{ij}]_{i,j=\overline{1,n}}.
\end{align*}
\end{enumerate}
The aforementioned methods are tested on randomly generated datasets with different sizes. To be more specific, an $n\times n$ matrix $A$ and a vector $b\in\R^n$ are generated randomly with i.i.d. (independent and identically distributed) standard Gaussian entries. The behavior of the methods is investigated with different dimensions $n\in \set{50,100,200}$ and different noise levels $\varepsilon\in\set{0,10^{-8},10^{-4},10^{-2}}$. The noise level $\varepsilon=0$ is considered as the noiseless case.  In total, the methods are tested in 24 different problems. All the methods are executed until they reach the maximum number of function evaluations of $200^*n$. The results in terms of function values achieved by the methods are presented in the Figures 
\ref{fig:C1L 50 result}--\ref{fig:C1L 200 result} below, where ``$L$" stands for least square, ``$N$" stands for nonconvex image restoration, and the number after ```$L$" and ``$N$" is the dimension of the problem.  It can be seen that all 24 out of 24 test problems admit our methods (DFC-fordif or DFC-cendif) as {\em the best} among the tested algorithms.
\begin{figure}
\centering
\includegraphics[width=.45\textwidth]{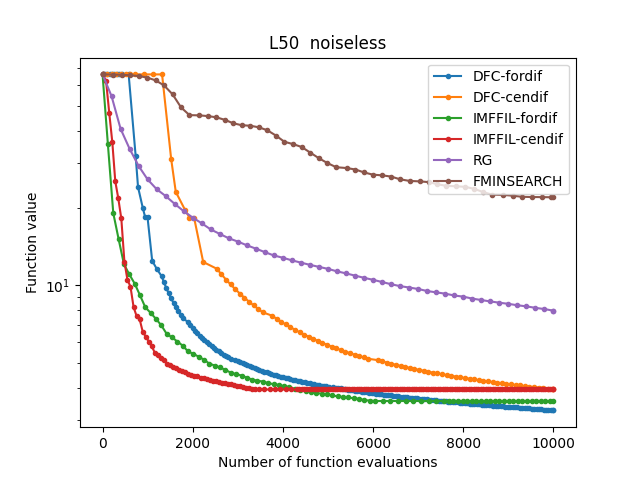}\quad
\includegraphics[width=.45\textwidth]{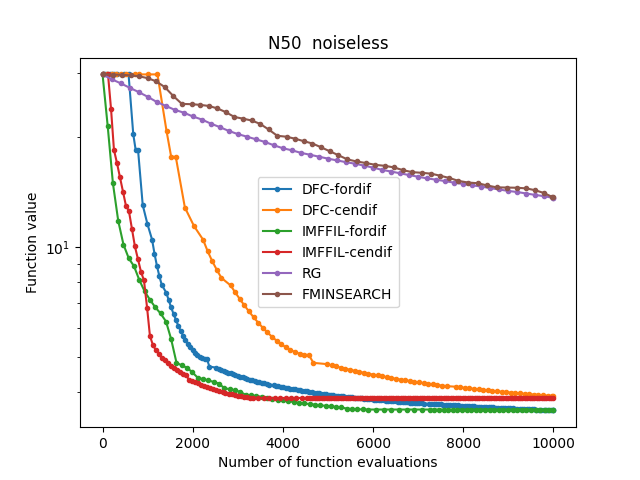}\quad

\includegraphics[width=.45\textwidth]{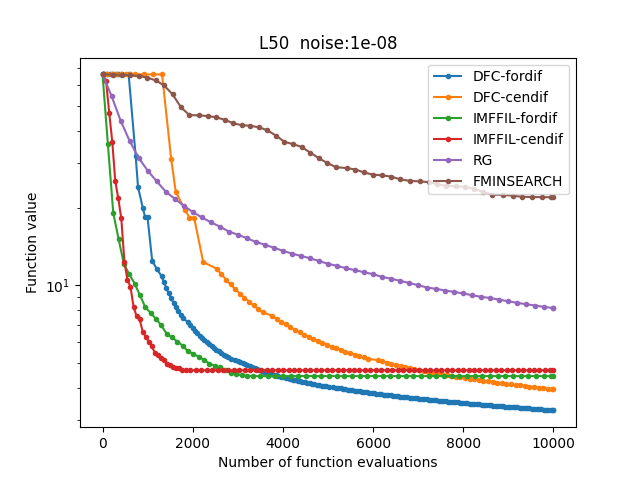}\quad
\includegraphics[width=.45\textwidth]{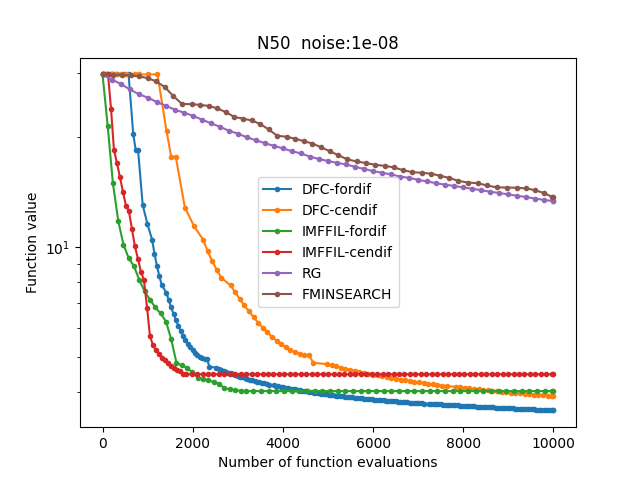}\quad

\includegraphics[width=.45\textwidth]{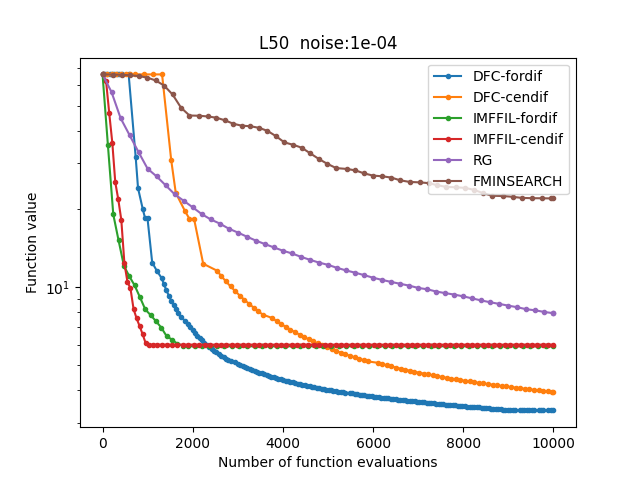}\quad
\includegraphics[width=.45\textwidth]{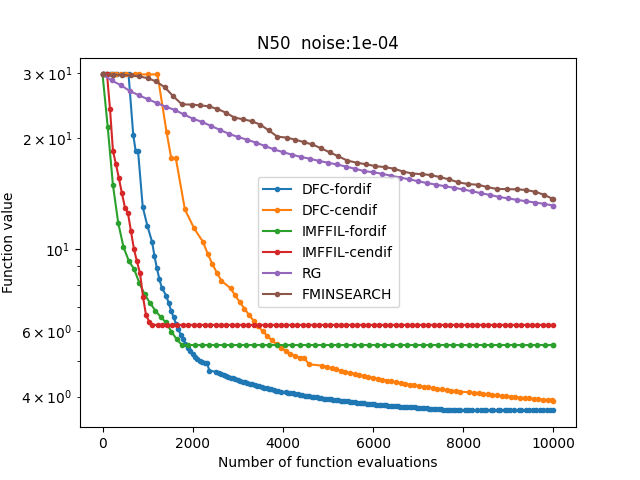}\quad

\includegraphics[width=.45\textwidth]{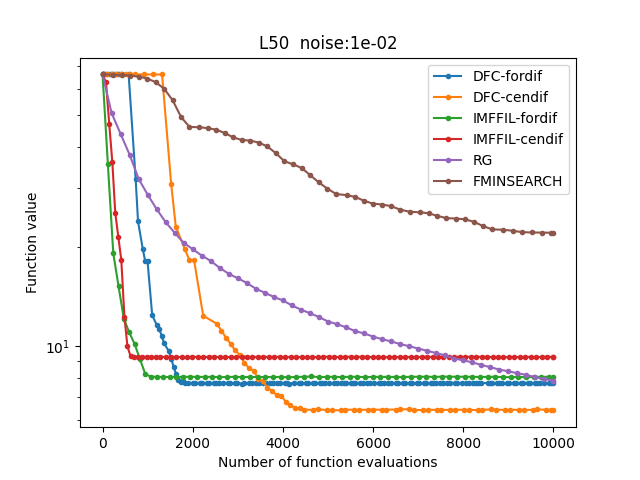}\quad
\includegraphics[width=.45\textwidth]{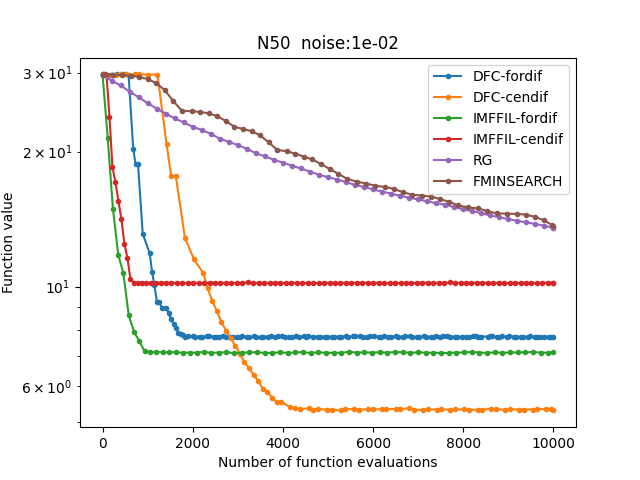}\quad
 \caption{$C^{1,1}_L$ functions with dimension $n=50$}\label{fig:C1L 50 result}
\end{figure}

\begin{figure}
\centering
\includegraphics[width=.45\textwidth]{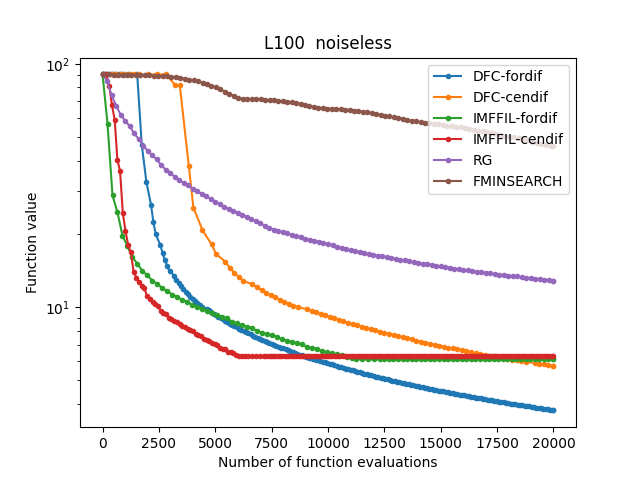}\quad
\includegraphics[width=.45\textwidth]{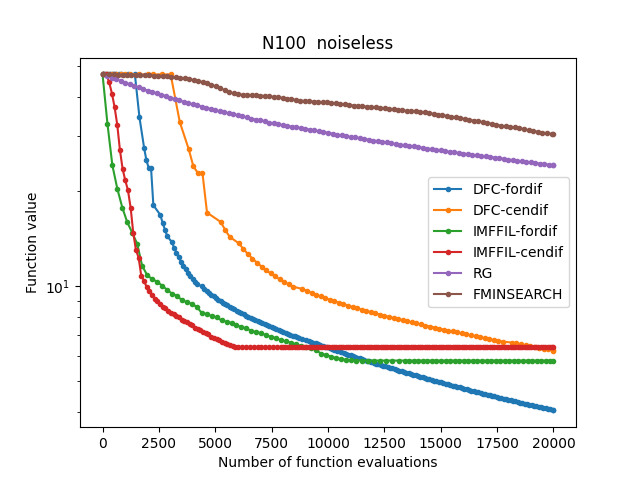}\quad

\includegraphics[width=.45\textwidth]{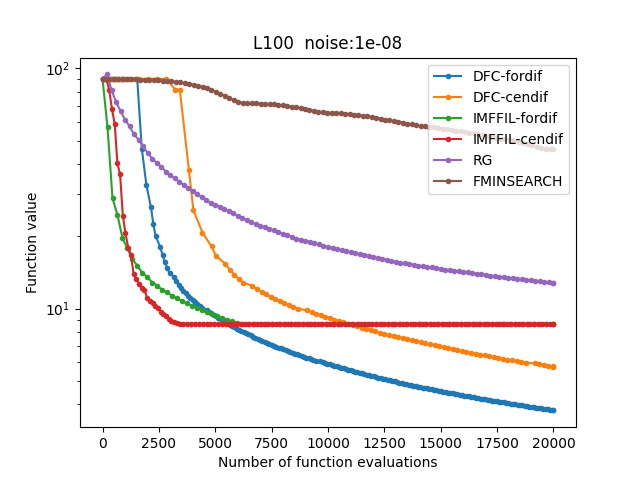}\quad
\includegraphics[width=.45\textwidth]{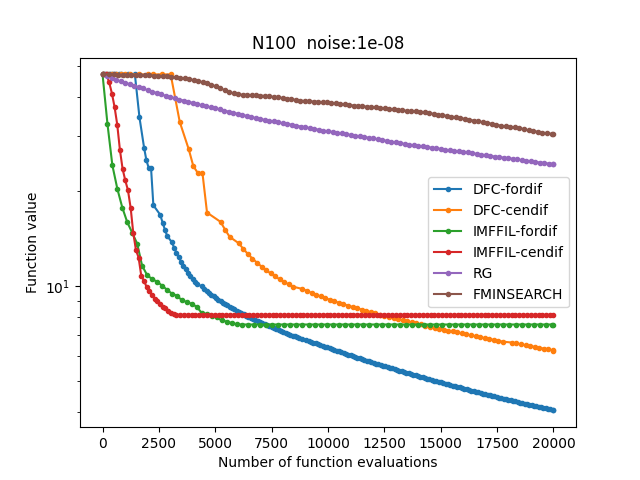}\quad

\includegraphics[width=.45\textwidth]{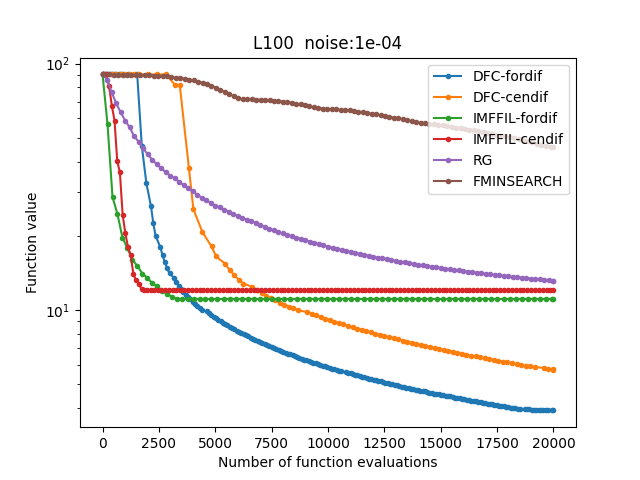}\quad
\includegraphics[width=.45\textwidth]{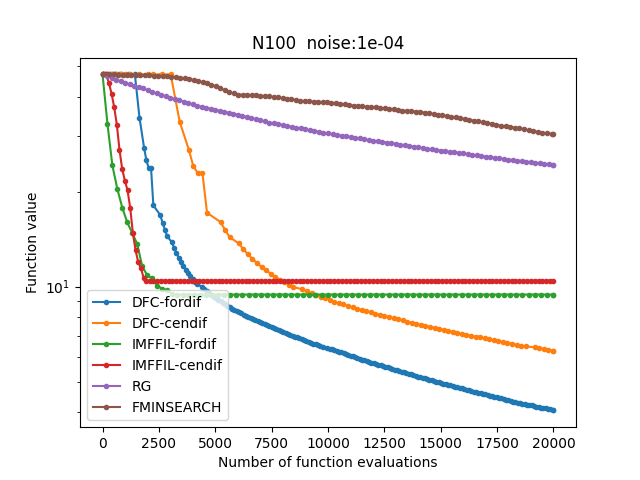}\quad

\includegraphics[width=.45\textwidth]{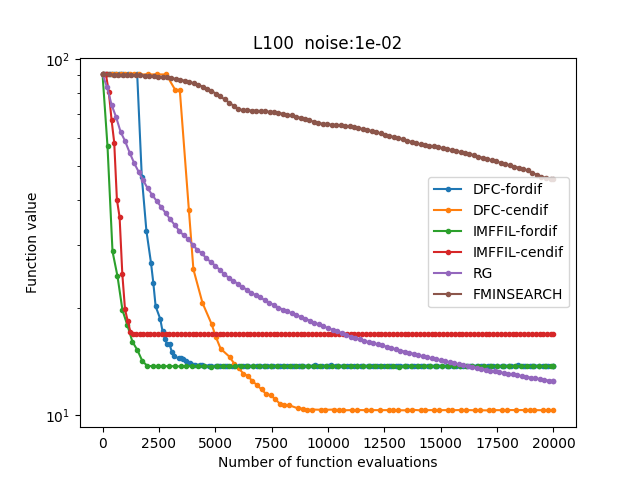}\quad
\includegraphics[width=.45\textwidth]{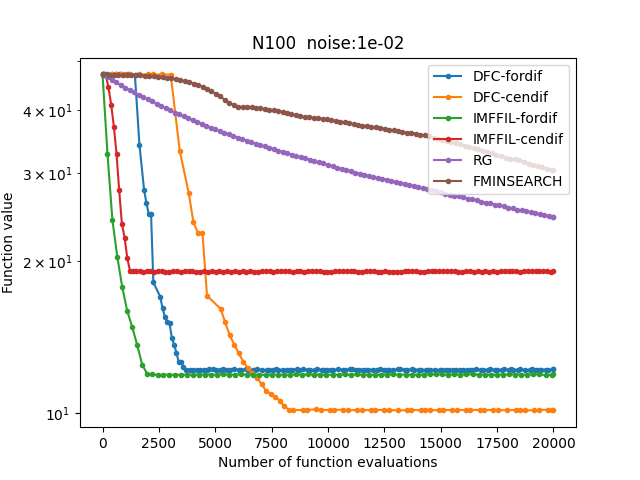}\quad
 \caption{$C^{1,1}_L$ functions with dimension $n=100$}\label{fig:C1L 100 result}
\end{figure}

\begin{figure}
\centering
\includegraphics[width=.45\textwidth]{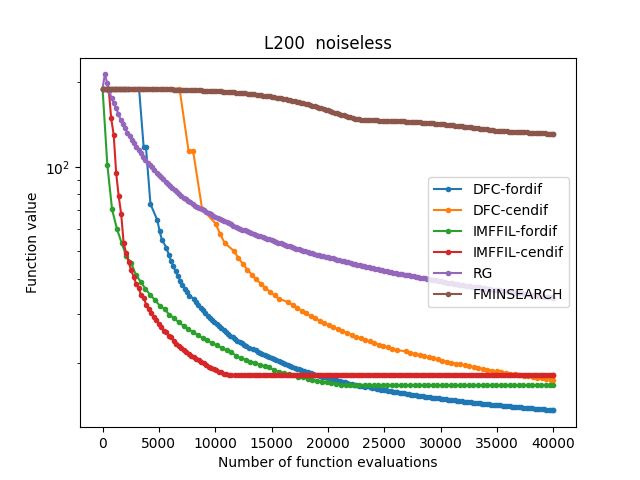}\quad
\includegraphics[width=.45\textwidth]{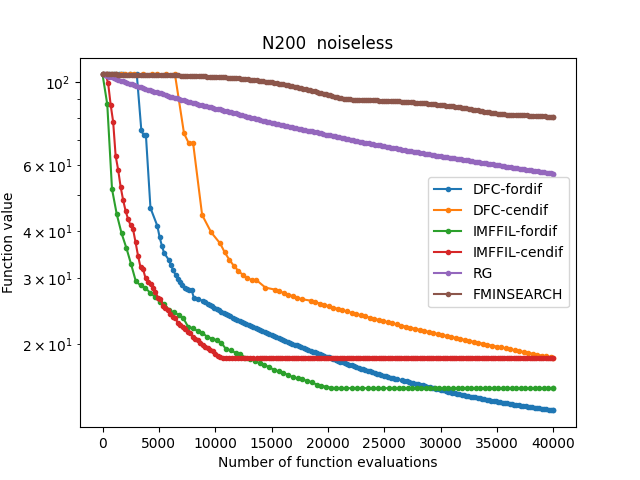}\quad

\includegraphics[width=.45\textwidth]{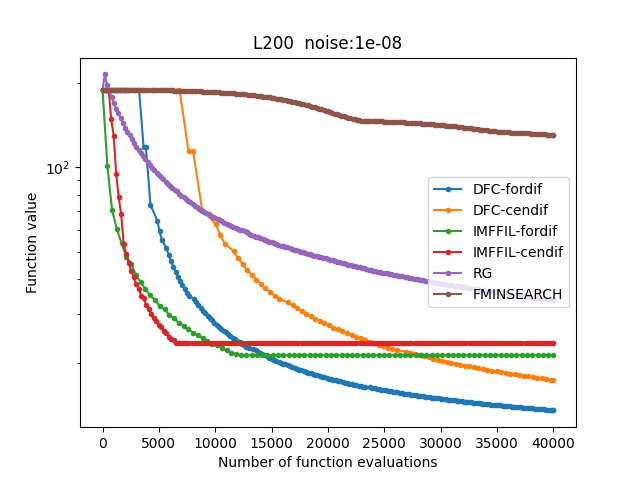}\quad
\includegraphics[width=.45\textwidth]{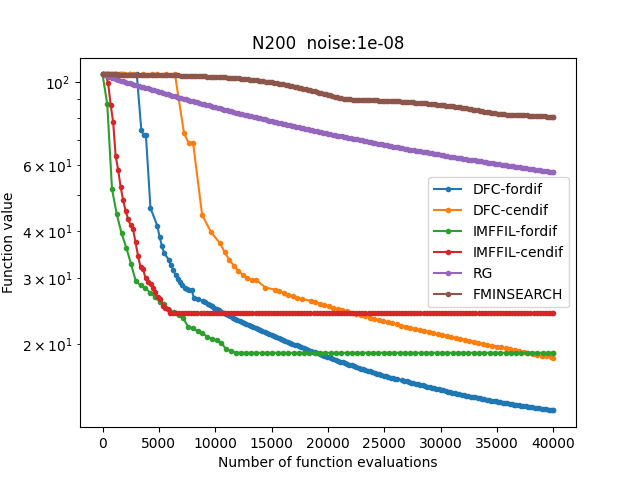}\quad

\includegraphics[width=.45\textwidth]{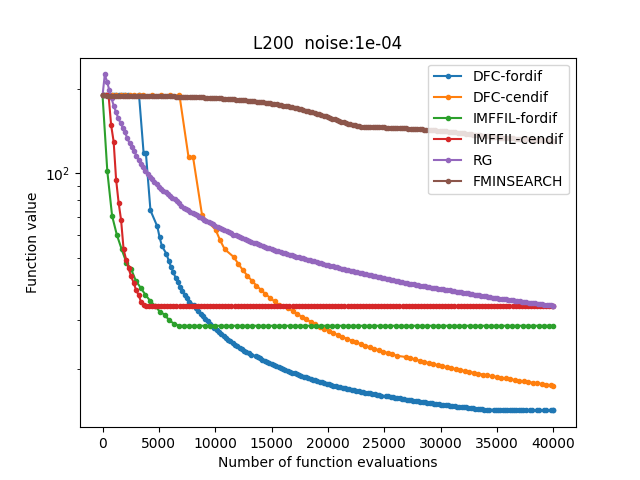}\quad
\includegraphics[width=.45\textwidth]{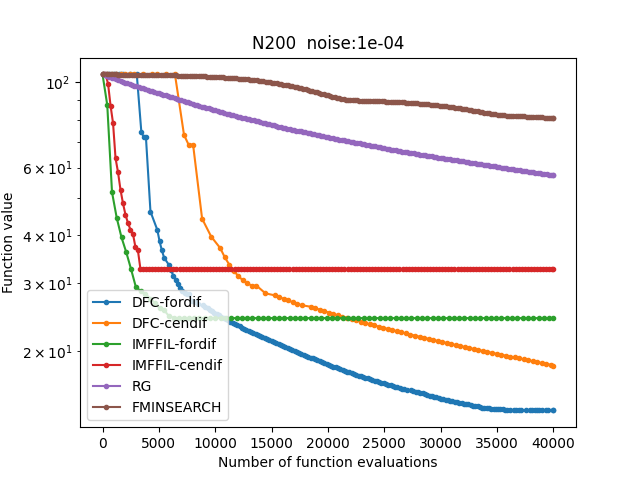}\quad

\includegraphics[width=.45\textwidth]{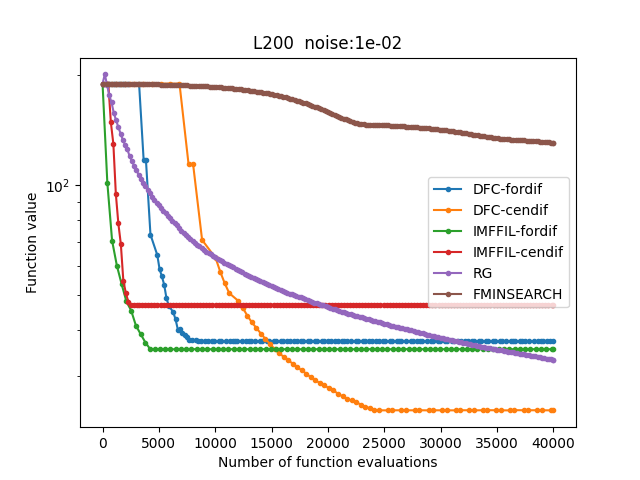}\quad
\includegraphics[width=.45\textwidth]{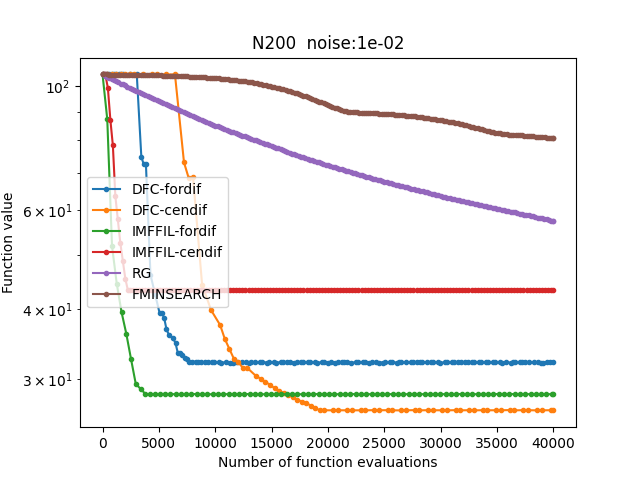}\quad
 \caption{$C^{1,1}_L$ functions with dimension $n=200$}\label{fig:C1L 200 result}
\end{figure}
\subsection{Experiments for \texorpdfstring{$\mathcal{C}^{1,1}$}{Lg} functions}\label{subsec 6.1}
This subsection is devoted to the comparison between the efficiency of our DFB methods, i.e., Algorithm~\ref{GDF C11} and that of other state-of-the-art derivative-free methods in solving $\mathcal{C}^{1,1}$ optimization problems. The methods considered in this numerical experiment are:
\begin{enumerate}[(i)]
\item DFB-fordif, i.e., Algorithm~\ref{GDF C11} with forward finite difference.
\item  DFB-cendif, i.e., Algorithm~\ref{GDF C11} with central finite difference.
\item FMINSEARCH, i.e., a Matlab implementation of the Nelder–Mead simplex-based method by Lagarias et al. \cite{lagarias98}.
\item  IMFIL-fordif, i.e., the implicit filtering algorithm with forward finite difference \cite{gilmore95}.
\item  IMFIL-cendif, i.e., the implicit filtering algorithm with central finite difference \cite{gilmore95}.
\end{enumerate}
It should be mentioned that the convergence properties of the implicit filtering algorithm in \cite{choi00,gilmore95,kelly99} are only valid for $\mathcal{C}^{1,1}_L$ functions under some rather strong conditions; see, e.g., \cite[Assumption~3.1]{gilmore95} and \cite[Theorem~2.1, Assumption~2.1]{choi00}, which are not applicable to $\mathcal{C}^{1,1}$ functions. However, this method is numerically tested here since its algorithm formulation does not require a Lipschitz constant of the gradient. The efficiency of all the aforementioned methods is compared in minimizing the noisy and nonconvex Rosenbrock function \cite{jamil13}, i.e., $\phi(x):=f(x)+\xi(x)$, where 
\begin{align*}
 f(x):=\sum_{i=1}^{n-1}\left[100\left(x_{i+1}-x_{i}^{2}\right)^{2}+\left(x_{i}-1\right)^{2}\right],\quad x\in\R^n,
\end{align*}
and where $\xi(x)$ is a uniformly distributed random variable with a given noise level $\varepsilon\ge 0$, i.e., $\xi(x)\sim U(-\varepsilon,\varepsilon)$. We investigate the behavior of the methods for different dimensions $n\in \set{50,100,200}$, different noise levels $\varepsilon\in\set{0,10^{-8
},10^{-4},10^{-2}}$. The noise level $0$ is considered as the noiseless case. Since the function $f$ has different local minimizers, the methods are also tested with two different initial points which are either the zero vector or the vector containing all $0.5$ elements. In total, the methods are tested in 24 problems. All the methods are run until they reach the maximum number of function evaluations of $200^*n$. It can be seen in Figures~\ref{fig:50 result}--\ref{fig:200 result} below that 20 problems admit our methods (DFB-fordif or DFB-cendif) as {\em the best} among tested algorithms, {\em except} for $4$ problems with the noise level $10^{-2}.$ 
\begin{figure}[H]
\centering
\includegraphics[width=.45\textwidth]{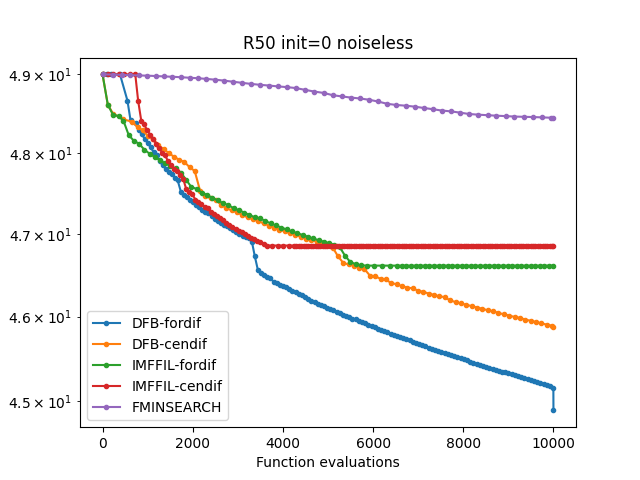}\quad
\includegraphics[width=.45\textwidth]{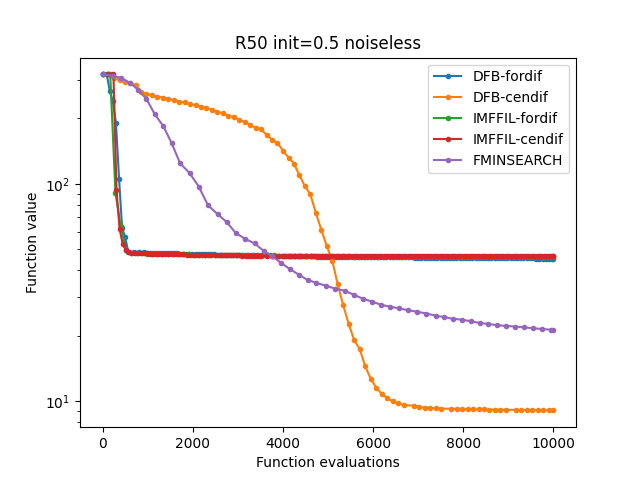}\quad

\includegraphics[width=.45\textwidth]{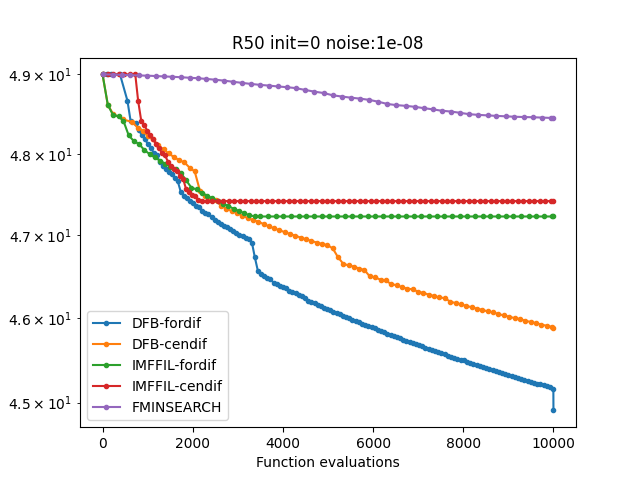}\quad
\includegraphics[width=.45\textwidth]{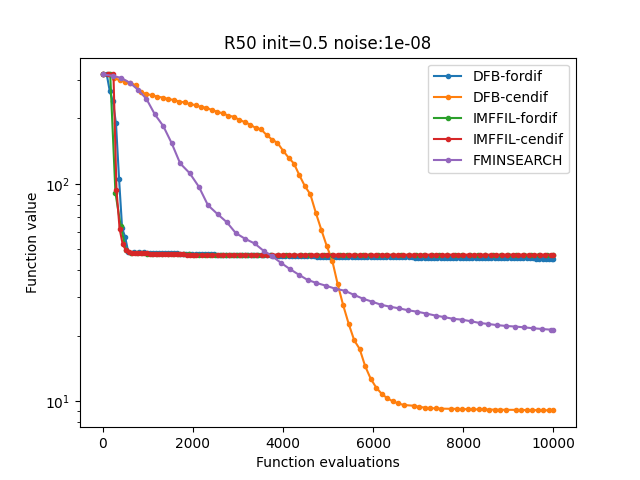}\quad

\includegraphics[width=.45\textwidth]{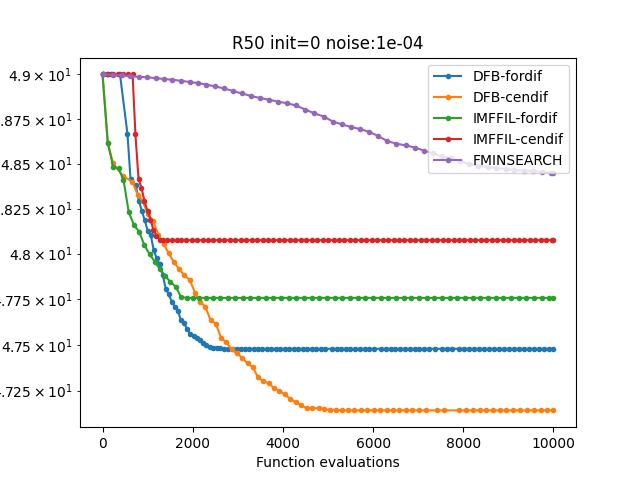}\quad
\includegraphics[width=.45\textwidth]{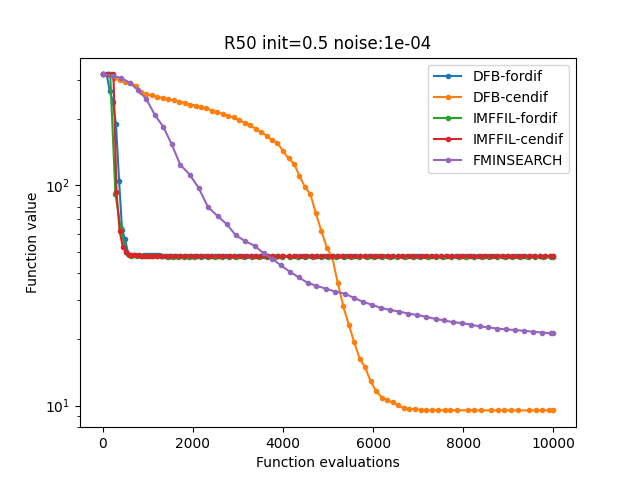}\quad

\includegraphics[width=.45\textwidth]{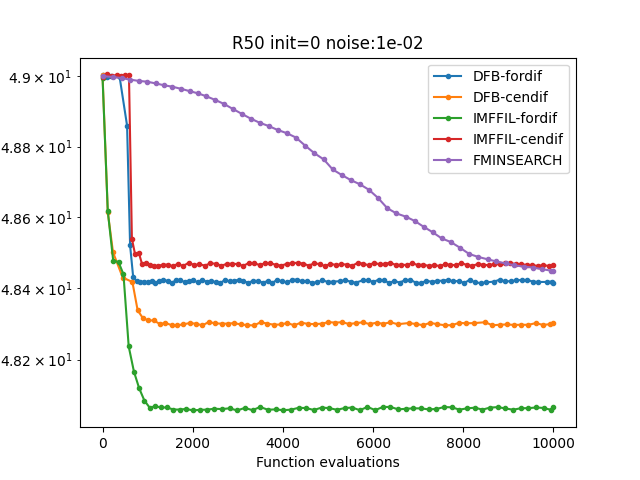}\quad
\includegraphics[width=.45\textwidth]{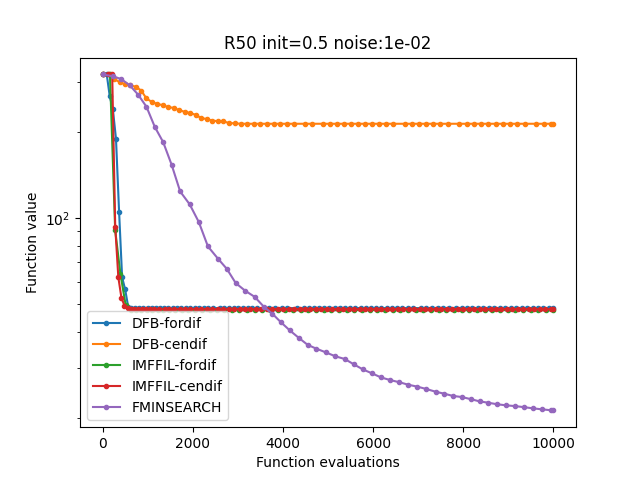}\quad
 \caption{$C^{1,1}$ functions with dimension $n=50$}\label{fig:50 result}
\end{figure}\vspace*{-0.2in}

\begin{figure}[H]
\centering
\includegraphics[width=.45\textwidth]{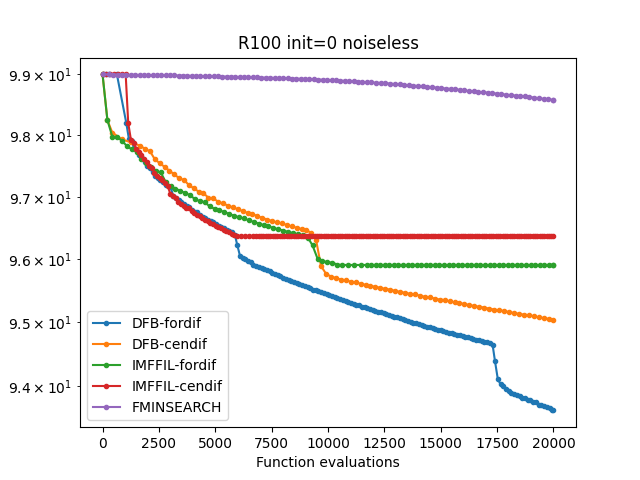}\quad
\includegraphics[width=.45\textwidth]{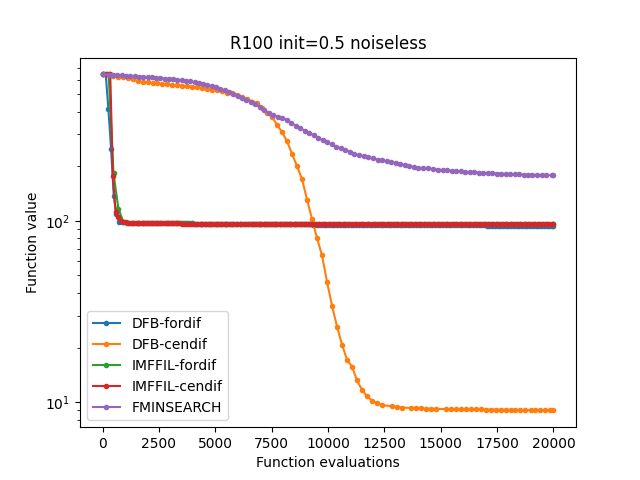}\quad

\includegraphics[width=.45\textwidth]{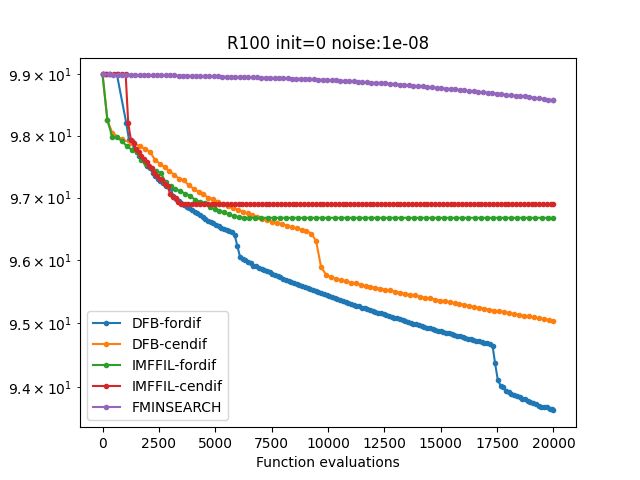}\quad
\includegraphics[width=.45\textwidth]{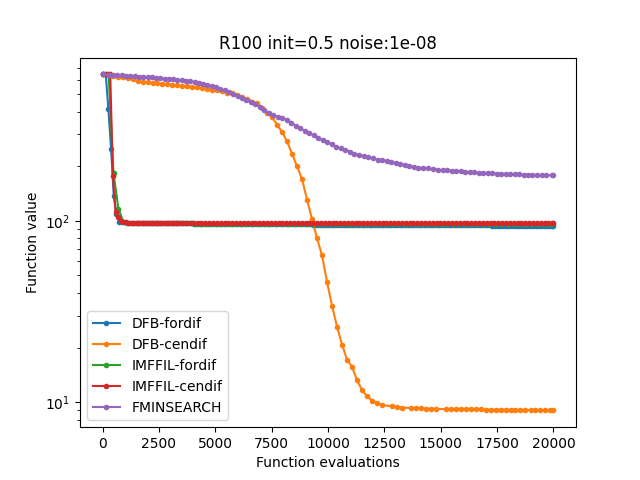}\quad

\includegraphics[width=.45\textwidth]{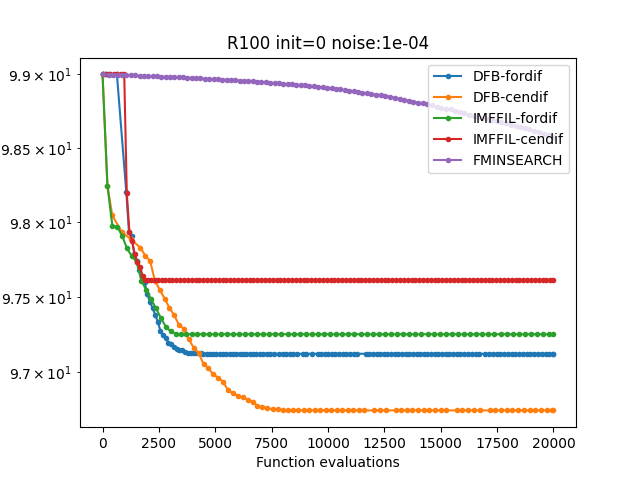}\quad
\includegraphics[width=.45\textwidth]{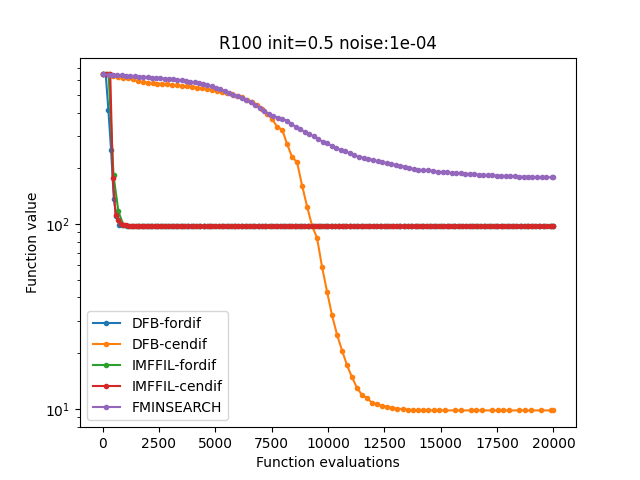}\quad

\includegraphics[width=.45\textwidth]{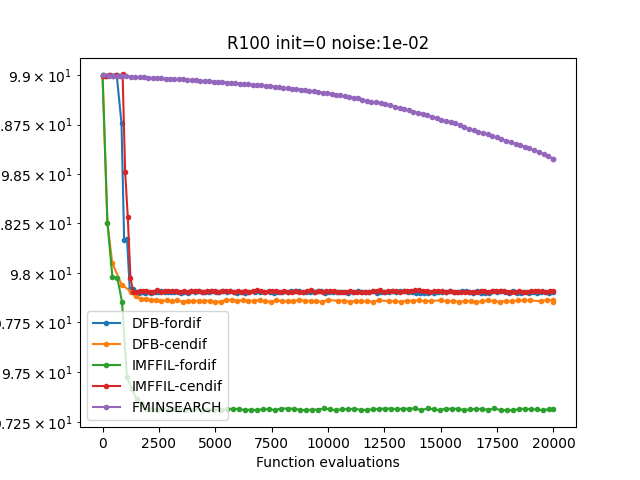}\quad
\includegraphics[width=.45\textwidth]{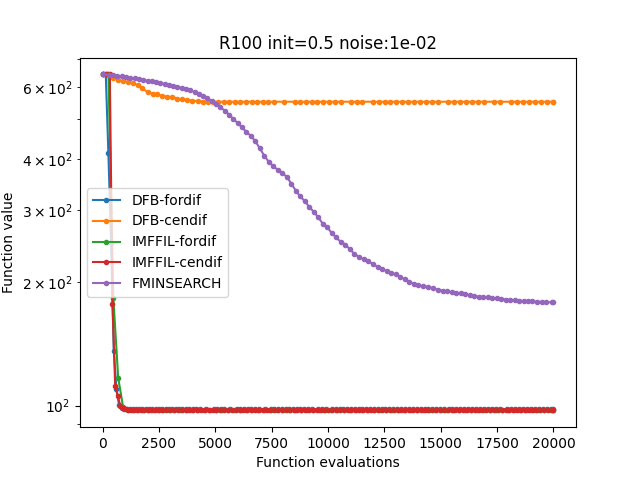}\quad
 \caption{$C^{1,1}$ functions with dimension $n=100$}\label{fig:100 result}
\end{figure}\vspace*{-0.2in}

\begin{figure}[H]
\centering
\includegraphics[width=.45\textwidth]{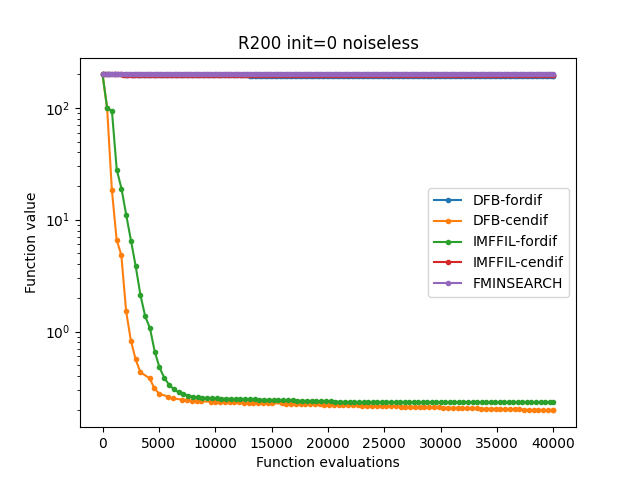}\quad
\includegraphics[width=.45\textwidth]{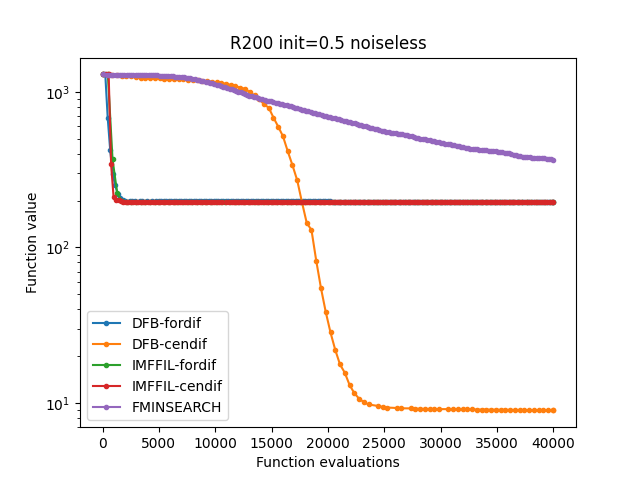}\quad

\includegraphics[width=.45\textwidth]{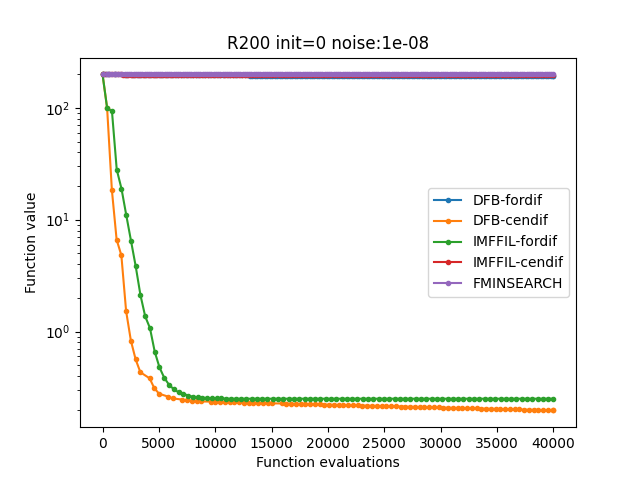}\quad
\includegraphics[width=.45\textwidth]{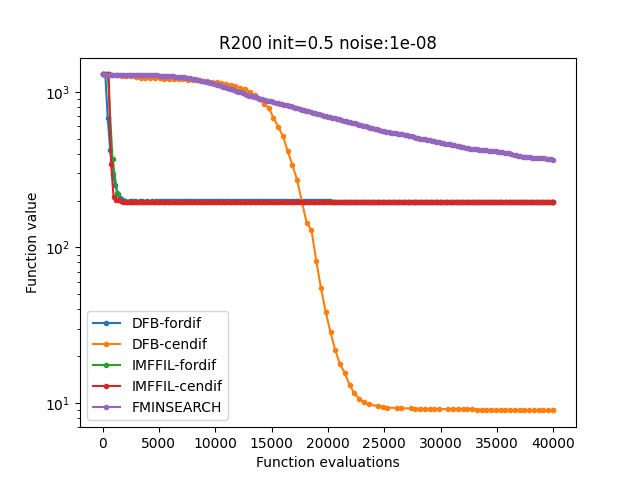}\quad

\includegraphics[width=.45\textwidth]{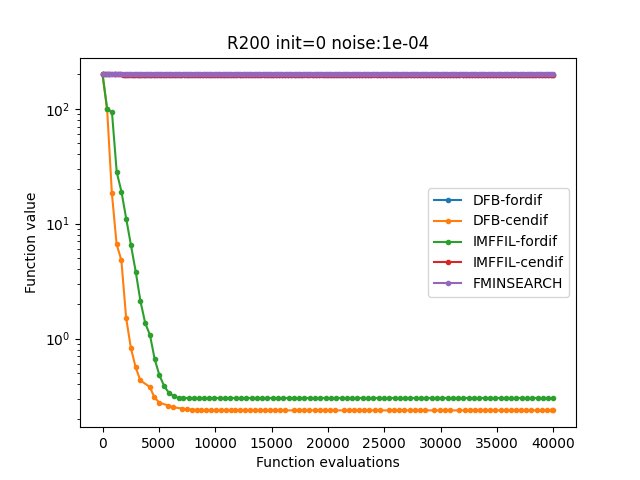}\quad
\includegraphics[width=.45\textwidth]{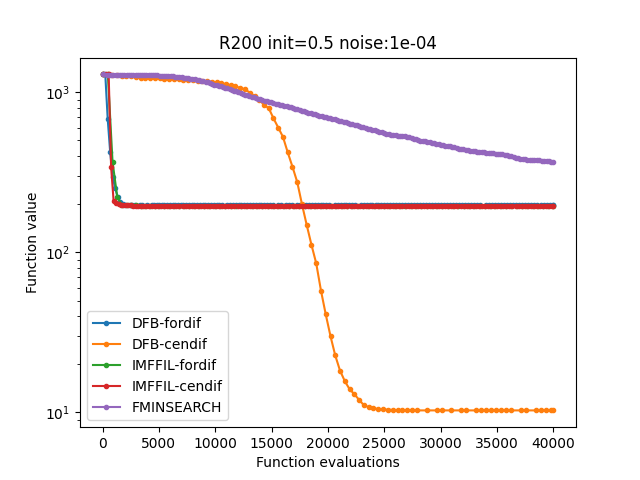}\quad

\includegraphics[width=.45\textwidth]{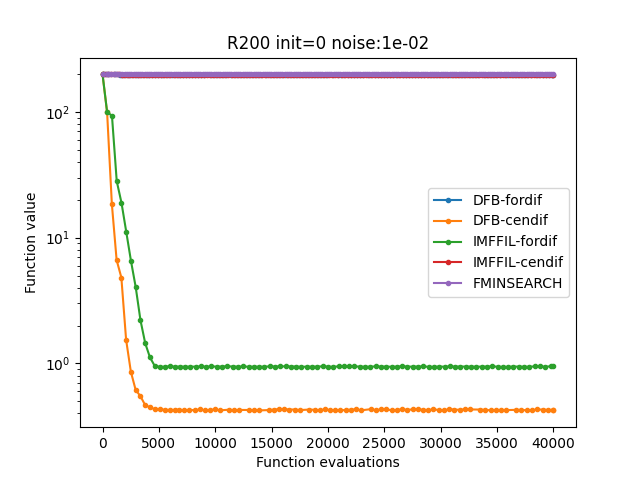}\quad
\includegraphics[width=.45\textwidth]{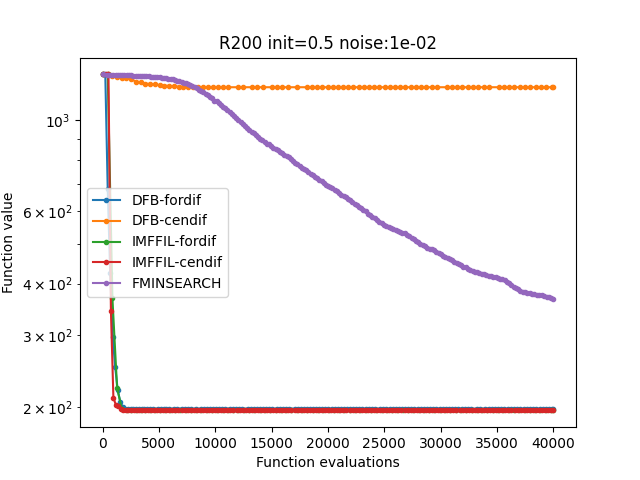}\quad
 \caption{$C^{1,1}$ functions with dimension $n=200$}\label{fig:200 result}
\end{figure}

\section{Concluding Remarks}\label{sec:concl}
This paper proposes two general derivative-free optimization methods for minimizing smooth functions. The proposed methods achieve stationary accumulation points, and under the KL property of the objective function, the global convergence with constructive convergence rates. The two newly developed algorithms cover derivative-free methods based on finite differences for $\mathcal{C}^{1,1}_L$ and $\mathcal{C}^{1,1}$ functions, where the finite difference intervals are chosen as large as possible while being automatically adapted with the exact gradients without knowing them. Our methods do not force finite difference intervals to decrease after each iteration, which omits rounding errors as much as possible and thus leads to better numerical performances in solving general convex, nonconvex, noiseless, and noisy derivative-free smooth problems in finite-dimensional spaces. In addition to the aforementioned globally convergent algorithms, we design their local versions, which locally converge to nonisolated local minimizers.

Our future research includes convergence analysis of the newly developed algorithms coupled with quasi-Newton methods for noisy smooth functions. We also intend to establish efficient conditions to ensure local and global convergence to local minimizers of iterative sequences generated by derivative-free methods for problems of nonsmooth constrained optimization.
\end{document}